\documentclass[12pt]{amsart}

\usepackage{graphicx, amsmath, amssymb}
\usepackage{fullpage}
\usepackage{tikz}
\usepackage{tkz-berge}
\usepackage{tkz-graph}
\usetikzlibrary{decorations}
\usepackage{etoolbox}

\newtheorem{thm}{Theorem}[section]
\newtheorem{prop}[thm]{Proposition}

\newtheorem{lemma}[thm]{Lemma}
\newtheorem{cor}[thm]{Corollary}
\theoremstyle{definition}
\newtheorem{definition}[thm]{Definition}

\usepackage{tikz}

\newcommand{\atipsm}{15}
\newcommand{\rin}{3}
\newcommand{\rout}{4}

\newcounter{ZZ}
\newcommand{\arcarrow}[4]{%
  \pgfmathsetmacro{\astart}{#1}
  \pgfmathsetmacro{\adiff}{#2}
  \pgfmathsetmacro{\plabel}{#4}
  \pgfmathsetmacro{\aend}{\astart+\adiff}
  \fill[#3] (\astart+\atipsm:\rin) arc (\astart+\atipsm:\aend-\atipsm:\rin)
       -- (\aend:\rout) arc (\aend:\astart:\rout)
       -- cycle;

  \coordinate  (A\plabel) at (\astart+\atipsm:\rin);
  \coordinate  (B\plabel) at (\aend-\atipsm:\rin);
  \filldraw(A\plabel) circle (6pt);
  \filldraw(B\plabel) circle (6pt);
}

\newcommand{\smallarcarrow}[4]{%
  
  \pgfmathsetmacro{\astart}{#1}
  \pgfmathsetmacro{\adiff}{#2}
  \pgfmathsetmacro{\plabel}{#4}
  \pgfmathsetmacro{\aend}{\astart+\adiff+\adiff}
  \fill[#3] (\astart+\atipsm:\rin) arc (\astart+\atipsm:\aend-\atipsm:\rin)
       -- (\aend:\rout) arc (\aend:\astart:\rout)
       -- cycle;

  \coordinate  (A\plabel) at (\astart+\adiff:\rin);
  \filldraw(A\plabel) circle (6pt);
    
}

\newcommand{\leveltwo}[1]{
\begin{tikzpicture}[rotate=0,scale=#1]
\setcounter{ZZ}{0}
\foreach \x in {0,180} {
   \arcarrow{\x}{180}{blue!15,draw = blue!50!black}{\value{ZZ}} 
   \stepcounter{ZZ}
}

  \draw[thick](A1) -- (A0);
  \draw[thick](B0) -- (B1);
  \draw[thick, blue!60](A0) -- (30:\rout) -- (45:\rin)--(60:\rout)--(70:\rout-.5);
\draw[thick,blue!60](B0) -- (150:\rout) -- (135:\rin) -- (120:\rout)-- (110:\rout-.5) ;
\draw[thick,blue!60](A1) -- (-150:\rout) -- (-135:\rin) -- (-120:\rout)-- (-110:\rout-.5) ;
\draw[thick,blue!60](B1)-- (-30:\rout) -- (-45:\rin)--(-60:\rout)--(-70:\rout-.5);

 \foreach \x in {45, -45,  135, -135}{
  \filldraw[blue!40!black](\x:\rin) circle (6pt);
  };
  
   \foreach \x in {0,30,60,120,150,-30,-60,-120,-150,180}{
  \filldraw[blue!50!black](\x:\rout) circle (6pt);
  }
\end{tikzpicture}
}

\newcommand{\leveltworeduction}[1]{
\begin{tikzpicture}[rotate=0,scale=#1]
\coordinate (v1) at (2,0);
\coordinate (v2) at (-2,0);
\coordinate (v3) at (0,1);
\coordinate (v4) at (1,3);
\coordinate (v5) at (0,-1);
\coordinate (v6) at (1,-3);

\draw[thick, blue!60] (v2)--(v3) --(v1) --(v4) --(v2);
\draw[thick, blue!60] (v2)--(v5) --(v1) --(v6) --(v2);
\draw[thick, blue!60] (v3)--(v4);
\draw[thick, blue!60] (v5)--(v6);
\draw[thick] (v4)--(v5);
\draw[thick] (v3)--(v6);
\foreach \x in {1,2,3,4,5,6}{
\filldraw[blue!50!black] (v\x) circle (6pt);
};
\end{tikzpicture}
}

\newcommand{\nondescendant}[1]{

\begin{tikzpicture}[scale=#1]
\coordinate (v1) at (1,1);
\coordinate (v2) at (3,1);
\coordinate (v3) at (-1,1);
\coordinate (v4) at (-3,1);
\coordinate (v5) at (0,-1);
\coordinate (v6) at (0,3);
\coordinate (v7) at (2,3);

\draw[thick](v1)--(v2);
\draw[thick](v3)--(v4);
\draw[thick](v5) -- (v1) --(v6)--(v3)--(v5);
\draw[thick](v5) -- (v2) --(v6)--(v4)--(v5);
\draw[thick, blue!60!black!60] (v1) -- (v7) --(v2);
\draw[thick, blue!60!black!60] (v3) -- (v7) --(v4);
\foreach \x in {1,2,3,4,5,6,7}{
\filldraw[blue!50!black] (v\x) circle (6pt);
};

}
\newcommand{\levelthreea}[1]{
\begin{tikzpicture}[scale=#1]
\setcounter{ZZ}{0}
\foreach \x in {30,150,270} {
   \arcarrow{\x}{120}{blue!30,draw = blue!50!black}{\value{ZZ}} 
   \stepcounter{ZZ}
} 
  \foreach \x in {30,150,270}{
  \filldraw[blue!50!black](\x:\rout) circle (6pt);
  }
  \draw[thick](A2) -- (A0);
  \draw[thick](A1) -- (B2);
  \draw[thick](B0) -- (B1);
\end{tikzpicture}
}

\newcommand{\levelthreeb}[1]{
\begin{tikzpicture}[scale=#1]
\setcounter{ZZ}{0}
\foreach \x in {30,150,270} {
   \arcarrow{\x}{120}{blue!30,draw = blue!50!black}{\value{ZZ}} 
   \stepcounter{ZZ}
} 
  \foreach \x in {30,150,270}{
  \filldraw[blue!50!black](\x:\rout) circle (6pt);
  }
  \draw[thick](A0) -- (B1);
  \draw[thick](A1) -- (B2);
  \draw[thick](A2) -- (B0);
\end{tikzpicture}
}

\newcommand{\levelthreec}[1]{
\begin{tikzpicture}[scale=#1]
\setcounter{ZZ}{0}
\smallarcarrow{60}{30}{orange!20,draw = orange!50!black}{\value{ZZ}} 
\stepcounter{ZZ}
\foreach \x in {120,270} {
   \arcarrow{\x}{150}{blue!30,draw = blue!50!black}{\value{ZZ}} 
   \stepcounter{ZZ}
} 
  \foreach \x in  {60,120,270} {
  \filldraw[blue!50!black](\x:\rout) circle (6pt);
  }
  \draw[thick](A0) -- (B1);
  \draw[thick](A0) -- (A2);
  \draw[thick](A1) -- (B2);
\end{tikzpicture}
}

\newcommand{\fourbigzag}[9]{
\begin{tikzpicture}[scale=#9]
\setcounter{ZZ}{0}
\foreach \x in {45,135,225,315} {
\arcarrow{\x}{90}{blue!30,draw = blue!50!black}{\value{ZZ}} 
\stepcounter{ZZ}
} 

\foreach \x in {45,135,225,315} {
\filldraw[blue!50!black](\x:\rout) circle (6pt);
}
  \draw[thick](#1) -- (#2);
  \draw[thick](#3) -- (#4);
  \draw[thick](#5) -- (#6);
  \draw[thick](#7) -- (#8);
\end{tikzpicture}
}

\newcommand{\threebigzag}[9]{
\setcounter{ZZ}{1}
\begin{tikzpicture}[scale=#9]
\coordinate  (A0) at (72:\rout);
\coordinate  (B0) at (108:\rout);
\foreach \x in {108,216,324} {
\arcarrow{\x}{108}{blue!30,draw = blue!50!black}{\value{ZZ}} 
\stepcounter{ZZ}
} 
\foreach \x in  {108,216,324} {
\filldraw[blue!50!black](\x:\rout) circle (6pt);
}
\filldraw(A0) circle (6pt);
\filldraw(B0) circle (6pt);
\draw[thick](A0) -- (#1);
\draw[thick](A0) -- (#2);
\draw[thick](#3) -- (#4);
\draw[thick](#5) -- (#6);
\draw[thick](#7) -- (#8);
\end{tikzpicture}
}

\newcommand{\threebigzagosm}[9]{
\setcounter{ZZ}{1}
\begin{tikzpicture}[scale=#9]
\smallarcarrow{72}{18}{orange!20,draw = orange!50!black}{0} 
\foreach \x in {108,216,324} {
\arcarrow{\x}{108}{blue!30,draw = blue!50!black}{\value{ZZ}} 
\stepcounter{ZZ}
} 
\foreach \x in  {72,108,216,324} {
\filldraw[blue!50!black](\x:\rout) circle (6pt);
}
\draw[thick](#1) -- (#2);
\draw[thick](#3) -- (#4);
\draw[thick](#5) -- (#6);
\draw[thick](#7) -- (#8);
\end{tikzpicture}
}

\newcommand{\twobigzagosm}[9]{
\begin{tikzpicture}[scale=#9]
\coordinate  (A0) at (70:\rout);
\coordinate  (B0) at (110:\rout);

\arcarrow{110}{140}{blue!30,draw = blue!50!black}{1} 
\smallarcarrow{250}{20}{orange!20,draw = orange!50!black}{2} 
\arcarrow{290}{140}{blue!30,draw = blue!50!black}{3} 
 
\foreach \x in  {250,290} {
  \filldraw[blue!50!black](\x:\rout) circle (6pt);
  }
  \draw[thick](A0) -- (#1);
  \draw[thick](A0) -- (#2);
  \draw[thick](#3) -- (#4);
  \draw[thick](#5) -- (#6);
  \draw[thick](#7) -- (#8);
  \filldraw(A0) circle (6pt);
  \filldraw(B0) circle (6pt);
\end{tikzpicture}
}

\newcommand{\twobigzagosmb}[9]{
\begin{tikzpicture}[scale=#9]
\coordinate  (B0) at (110:\rout);
\coordinate  (A0) at (70:\rout);
  
\smallarcarrow{110}{40}{orange!20,draw = orange!50!black}{1} 
\arcarrow{190}{120}{blue!30,draw = blue!50!black}{2} 
\arcarrow{310}{120}{blue!30,draw = blue!50!black}{3} 
 
\foreach \x in {110,190,310,70} {
  \filldraw[blue!50!black](\x:\rout) circle (6pt);
  }
\filldraw(A0) circle (6pt);
\filldraw(B0) circle (6pt);
  \draw[thick](A0) -- (#1);
  \draw[thick](A0) -- (#2);
  \draw[thick](#3) -- (#4);
  \draw[thick](#5) -- (#6);
  \draw[thick](#7) -- (#8);
\end{tikzpicture}
}

\newcommand{\twolongtwoshorta}[1]{
\begin{tikzpicture}[scale=#1]
\smallarcarrow{70}{20}{orange!20,draw = orange!50!black}{0} 
\arcarrow{110}{140}{blue!30,draw = blue!50!black}{1} 
\smallarcarrow{250}{20}{orange!20,draw = orange!50!black}{2} 
\arcarrow{290}{140}{blue!30,draw = blue!50!black}{3} 
 
\foreach \x in {70,110,250,290} {
  \filldraw[blue!50!black](\x:\rout) circle (6pt);
}
\draw[thick](A0) -- (B1);
\draw[thick](A0) -- (A3);
\draw[thick](A2) -- (A1);
\draw[thick](A2) -- (B3);
\end{tikzpicture}
}
\newcommand{\twolongtwoshortb}[1]{
\begin{tikzpicture}[scale=#1]
\smallarcarrow{70}{20}{orange!20,draw = orange!50!black}{0} 
\arcarrow{110}{140}{blue!30,draw = blue!50!black}{1} 
\smallarcarrow{250}{20}{orange!20,draw = orange!50!black}{2} 
\arcarrow{290}{140}{blue!30,draw = blue!50!black}{3} 
 
\foreach \x in {70,110,250,290} {
  \filldraw[blue!50!black](\x:\rout) circle (6pt);
}
\draw[thick](A0) -- (B1);
\draw[thick](A0) -- (A2);
\draw[thick](A1) -- (A3);
\draw[thick](A2) -- (B3);
\end{tikzpicture}
}
\newcommand{\twolongtwoshortc}[1]{
\begin{tikzpicture}[scale=#1]
\smallarcarrow{50}{20}{orange!20,draw = orange!50!black}{0} 
\smallarcarrow{90}{20}{orange!20,draw = orange!50!black}{1} 
\arcarrow{130}{140}{blue!30,draw = blue!50!black}{2} 
\arcarrow{270}{140}{blue!30,draw = blue!50!black}{3} 
 
\foreach \x in {50,90,130,270} {
  \filldraw[blue!50!black](\x:\rout) circle (6pt);
}
\draw[thick](A0) -- (A3);
\draw[thick](A0) -- (A2);
\draw[thick](A1) -- (B2);
\draw[thick](A1) -- (B3);
\end{tikzpicture}
}

\newcommand{\onebigzigzaga}[1]{
\begin{tikzpicture}[scale=#1]
\coordinate  (A0) at (70:\rout);
\coordinate  (B0) at (110:\rout);
\smallarcarrow{110}{40}{orange!20,draw = orange!50!black}{1} 
\smallarcarrow{190}{40}{orange!20,draw = orange!50!black}{2} 
\arcarrow{270}{160}{blue!30,draw = blue!50!black}{3} 
\foreach \x in {110,190,270} {
  \filldraw[blue!50!black](\x:\rout) circle (6pt);
}
\filldraw(A0) circle (6pt);
\filldraw(B0) circle (6pt);
\draw[thick](A0) -- (A1);
\draw[thick](A0) -- (A2);
\draw[thick](B0) -- (B3);
\draw[thick](B0) -- (A2);
\draw[thick](A1) -- (A3);
\end{tikzpicture}
}

\newcommand{\onebigzigzagb}[1]{
\begin{tikzpicture}[scale=#1]
\coordinate  (A0) at (70:\rout);
\coordinate  (B0) at (110:\rout);
\smallarcarrow{110}{40}{orange!20,draw = orange!50!black}{1} 
\smallarcarrow{190}{40}{orange!20,draw = orange!50!black}{2} 
\arcarrow{270}{160}{blue!30,draw = blue!50!black}{3} 
\foreach \x in {110,190,270} {
  \filldraw[blue!50!black](\x:\rout) circle (6pt);
}
\filldraw(A0) circle (6pt);
\filldraw(B0) circle (6pt);
\draw[thick](A0) -- (A1);
\draw[thick](A0) -- (A2);
\draw[thick](B0) -- (A2);
\draw[thick](B0) -- (A3);
\draw[thick](A1) -- (B3);
\end{tikzpicture}
}
\newcommand{\onebigzigzagc}[1]{
\begin{tikzpicture}[scale=#1]
\coordinate  (A0) at (70:\rout);
\coordinate  (B0) at (110:\rout);
\smallarcarrow{110}{40}{orange!20,draw = orange!50!black}{1} 
\arcarrow{190}{160}{blue!30,draw = blue!50!black}{2} 
\smallarcarrow{350}{40}{orange!20,draw = orange!50!black}{3} 

\filldraw(A0) circle (6pt);
\filldraw(B0) circle (6pt);
\draw[thick](A0) -- (A1);
\draw[thick](A0) -- (A2);
\draw[thick](B0) -- (B2);
\draw[thick](B0) -- (A3);
\draw[thick](A1) -- (A3);
\foreach \x in {110,190,350} {
  \filldraw[blue!50!black](\x:\rout) circle (6pt);
}
\end{tikzpicture}
}
\newcommand{\onebigzigzagd}[1]{
\begin{tikzpicture}[scale=#1]
\coordinate  (A0) at (70:\rout);
\coordinate  (B0) at (110:\rout);
\smallarcarrow{110}{40}{orange!20,draw = orange!50!black}{1} 
\arcarrow{190}{160}{blue!30,draw = blue!50!black}{2} 
\smallarcarrow{350}{40}{orange!20,draw = orange!50!black}{3} 
\foreach \x in {110,190,350} {
  \filldraw[blue!50!black](\x:\rout) circle (6pt);
}
\filldraw(A0) circle (6pt);
\filldraw(B0) circle (6pt);
\draw[thick](A0) -- (A1);
\draw[thick](A0) -- (B2);
\draw[thick](B0) -- (A2);
\draw[thick](B0) -- (A3);
\draw[thick](A1) -- (A3);
\end{tikzpicture}
}
\newcommand{\nobigzigzag}[1]{
\setcounter{ZZ}{1}
\begin{tikzpicture}[scale=#1]
\coordinate  (A0) at (72:\rout);
\coordinate  (B0) at (108:\rout);
\foreach \x in {108,216,324} {
\smallarcarrow{\x}{54}{orange!20,draw = orange!50!black}{\value{ZZ}} 
\stepcounter{ZZ}
}
\foreach \x in {108,216,324} {
  \filldraw[blue!50!black](\x:\rout) circle (6pt);
}
\filldraw(A0) circle (6pt);
\filldraw(B0) circle (6pt);
\draw[thick](A0) -- (A1);
\draw[thick](A0) -- (A2);
\draw[thick](B0) -- (A2);
\draw[thick](B0) -- (A3);
\draw[thick](A1) -- (A3);
\end{tikzpicture}
}

\usepackage{multirow}
\usepackage{comment}
\usepackage{float}
\usepackage{colortbl}
\usepackage{subcaption}

\usetikzlibrary{graphs, graphs.standard,quotes}
\usetikzlibrary{shapes.geometric,positioning}

\newcommand{\CV}{\operatorname{CV}}

\newcounter{MarginCounter}

\newcommand{\mrgn}[1]{\stepcounter{MarginCounter}\textcolor{red}{$^{\arabic{MarginCounter}}$}\marginpar{\textcolor{red}{\tiny{$^{\arabic{MarginCounter}}$ #1}}}}
\renewcommand{\mrgn}[1]{}
\newcommand{\ghostmrgn}[1]{\stepcounter{MarginCounter}}

\newtheoremstyle{case}{}{}{}{}{}{:}{ }{}
\theoremstyle{case}

\theoremstyle{plain}

\newcommand{\separator}{
	\hspace{0.5cm}
	\begin{tikzpicture}[thick]
	\node (v1) at (0,0) {};
	\draw [black,   -latex      ] (0,1) -- (1,1) node [right] {};
	\end{tikzpicture}
	\hspace{0.5cm}
}


\title{Some results on double triangle descendants of $K_5$}
\author{Mohamed Laradji, Marni Mishna, and Karen Yeats}

\begin{document}
\begin{abstract}
        Double triangle expansion is an operation on $4$-regular graphs with at least one triangle which replaces a triangle with two triangles in a particular way.  We study the class of graphs which can be obtained by repeated double triangle expansion beginning with the complete graph $K_5$.  These are called double triangle descendants of $K_5$.  We enumerate, with explicit rational generating functions, those double triangle descendants of $K_5$ with at most four more vertices than triangles.  We also prove that the minimum number of triangles in any $K_5$ descendant is four.  Double triangle descendants are an important class of graphs because of conjectured properties of their Feynman periods when they are viewed as scalar Feynman diagrams, and also because of conjectured properties of their $c_2$ invariants, an arithmetic graph invariant with quantum field theoretical applications.
\end{abstract}
\maketitle

\section{Introduction}
\label{sec:intro}
In the study of Feynman integrals, it is useful to characterize the graphs that are obtainable by repeated application of a particular graph operation applied to the complete unlabelled graph on five vertices, $K_5$. The operation modifies triangles and each application is called a \emph{double triangle expansion}.  We call the class of graphs obtainable this way the \emph{double triangle descendants of $K_5$} or simply \emph{$K_5$ descendants}.  The motivation for studying these graphs comes from results and a conjecture in quantum field theory which we describe below. They connect the class of $K_5$ descendants to questions of transcendence in number theory. In this work we provide concrete results on the structure of this class, and describe how to enumerate its elements along a key parameter.

Double triangle expansion is a local graph modification, which transforms a triangle and a pendant edge into two triangles attached a pendant edge. More formally, suppose
$H$ is a subgraph of a graph $G$ with four distinct vertices $\{v_1, v_2, v_3, v_4\}$ in a particular configuration, specifically $H$ is composed of a triangle $\{v_1,v_2,v_3\}$ and a single additional edge $\{v_2,v_4\}$. In a \emph{double triangle expansion} of $G$, a new \emph{child} graph $G':=\mbox{DTE}_G(H)$ is created from the \emph{parent} $G$ by subdividing the edges $\{v_1,v_3\}$ and $\{v_2,v_4\}$ and identifying the two newly created vertices. The operation is illustrated in Figure~\ref{fig:DTEtriangleandanedge}, and an example is shown in Figure~\ref{fig:k5dte}.
	\label{defn:dte}

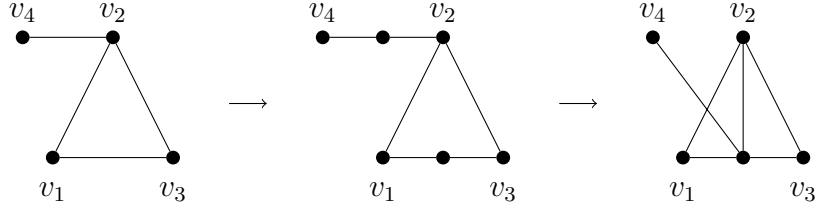
\begin{figure}
	\newcommand{\belowyshift}{-0.3cm}
	\centering
	\begin{tikzpicture}[scale=.8, every node/.style={draw,circle,very thick, fill=black, minimum size=4pt, inner sep=0pt}]
	\node[label=$v_2$] (v2) at (1,2) {};
	\node[label={[below,yshift=\belowyshift]$v_1$}] (v1) at (0,0) {};
	\node[label={[below,yshift=\belowyshift]$v_3$}] (v3) at (2,0) {};
	\node[label=$v_4$] (v4) at (-0.5,2) {};
	\draw
	(v1)--(v2)--(v3)--(v1)
	(v2)--(v4)
	;
	\end{tikzpicture}
	\begin{tikzpicture}
	\node (v1) at (0,0) {};
	\node (v2) at (1,2) {};
	\draw[->]
	(0.25,1.25)--(0.75,1.25);
	\end{tikzpicture}
	\begin{tikzpicture}[scale=.8,every node/.style={draw,circle,very thick, fill=black, minimum size=4pt, inner sep=0pt}]
	\node[label=$v_2$] (v2) at (1,2) {};
	\node[label={[below,yshift=\belowyshift]$v_1$}] (v1) at (0,0) {};
	\node[label={[below,yshift=\belowyshift]$v_3$}] (v3) at (2,0) {};
	\node[label=$v_4$] (v4) at (-1,2) {};
	\node (v5) at (0,2) {};
	\node (v6) at (1,0) {};
	\draw
	(v1)--(v2)--(v3)--(v6)--(v1)
	(v2)--(v5)--(v4)
	;
	\end{tikzpicture}
	\begin{tikzpicture}
	\node (v1) at (0,0) {};
	\node (v2) at (1,2) {};
	\draw[->]
	(0.25,1.25)--(0.75,1.25);
	\end{tikzpicture}
	\begin{tikzpicture}[scale=.8,every node/.style={draw,circle,very thick, fill=black, minimum size=4pt, inner sep=0pt}]
	\node[label=$v_2$] (v2) at (1,2) {};
	\node[label={[below,yshift=\belowyshift]$v_1$}] (v1) at (0,0) {};
	\node[label={[below,yshift=\belowyshift]$v_3$}] (v3) at (2,0) {};
	\node[label=$v_4$] (v4) at (-0.5,2) {};
	\node (v5) at (1,0) {};
	\draw
	(v1)--(v2)--(v3)--(v5)--(v1)
	(v2)--(v5)--(v4)
	;
	\end{tikzpicture}
	\caption[Double triangle expansion by subdividing]{Given a subgraph consisting of a triangle and a pendant edge, a DTE is applied by subdividing the edge not in the triangle and the opposite edge in the triangle, and identifying the newly created vertices.}
	\label{fig:DTEtriangleandanedge}
\end{figure}
\begin{figure}
	\centering
	\begin{tikzpicture}[scale = .6, every node/.style={draw,shape=circle,fill=black,scale=0.4}]
	\node (v1) at (0,0) {};
	\node (v2) at (1,2) {};
	\node (v3) at (2,0) {};
	\node (v4) at (3,2) {};
	\node (v5) at (4,0) {};
	\draw
	(v1)--(v2)--(v3)--(v4)--(v5)
	(v2)--(v4)
	(v1)--(v3)--(v5)
	[-] (v1)  to [out=120,in=-90,in looseness=1.5] (-0.2cm,0.5cm) to [out=90,in=135,out looseness=1.5](v4)
	[-] (v2)  to [out=45,in=90,in looseness=1.5] (4.2cm,0.5cm) to [out=-90,in=60,out looseness=1.5](v5)
	[-] (v1)  to [out=-30,in=180,in looseness=1] (2cm,-0.7cm) to [out=0,in=-150,out looseness=1](v5);

	\node[draw, color=white] at (6,1){$\rightarrow$};
	\end{tikzpicture}
	\begin{tikzpicture}[scale=.6, every node/.style={draw,shape=circle,fill=black,scale=0.4}]
	\node (v1) at (0,0) {};
	\node (v2) at (1,2) {};
	\node (v3) at (2,0) {};
	\node (v4) at (3,2) {};
	\node (v5) at (4,0) {};
	\node (v6) at (5,2) {};
	\draw
	(v1)--(v2)--(v3)--(v4)--(v5)--(v6)
	(v2)--(v4)--(v6)
	(v1)--(v3)--(v5)
	[-] (v1)  to [out=90,in=180,in looseness=1] (2cm,3cm) to [out=0,in=135,out looseness=1](v6)
	[-] (v2)  to [out=30,in=180,in looseness=1] (3cm,2.5cm) to [out=0,in=150,out looseness=1](v6)
	[-] (v1)  to [out=-30,in=180,in looseness=1] (2cm,-0.7cm) to [out=0,in=-150,out looseness=1](v5);
	\end{tikzpicture}
	\caption[Double triangle expansion on $K_5$]{The graph $\hat{Z}_3\cong K_5$ \emph{(left)} and $\hat{Z}_4$ \emph{(right)}. The graph $\hat{Z}_4$ can be obtained from $\hat{Z}_3$ by double triangle expansion of any triangle.}
	\label{fig:k5dte}
\end{figure}
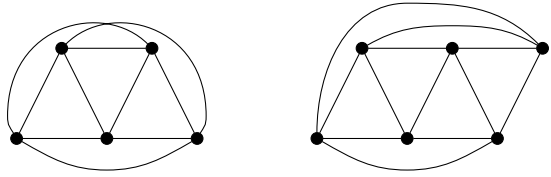

\subsection{Structure of the class of $K_5$ descendants}
In this work, we examine the set of graphs obtainable from $K_5$ by double triangle expansion. First we note that as this operation preserves the property of 4-regularity, each $K_5$-descendant is a 4-regular graph. Figure~\ref{fig:K5familytree} illustrates the graphs that are obtainable with at most 4 applications of this operation. 
In this article, we will show the slightly surprising result that the minimum number of triangles in a descendant of $K_5$ is $4$. A key structural parameter is the difference between the number of vertices of a graph, and the number of triangles. We are able to enumerate sub-families along this parameter, up to the case when the number of vertices is at most $4$ more than the number of triangles.  
The counting proofs also characterize the $K_5$ descendants of these levels by how their triangles are partitioned into \emph{zigzag} subgraphs.  Many of these results first appeared in Laradji's M.Sc thesis~\cite{LaradjiDTDoK5}.  The generating function for the $K_5$ descendants whose number of vertices is four more than the number of triangles that we find agrees with an earlier conjecture of Laradji. 

\newcommand{\bottomheight}{-2cm}
\newcommand{\ppbmult}{0.33}
\newcommand{\prebottommult}{0.20}
\newcommand{\maxwidth}{15cm}
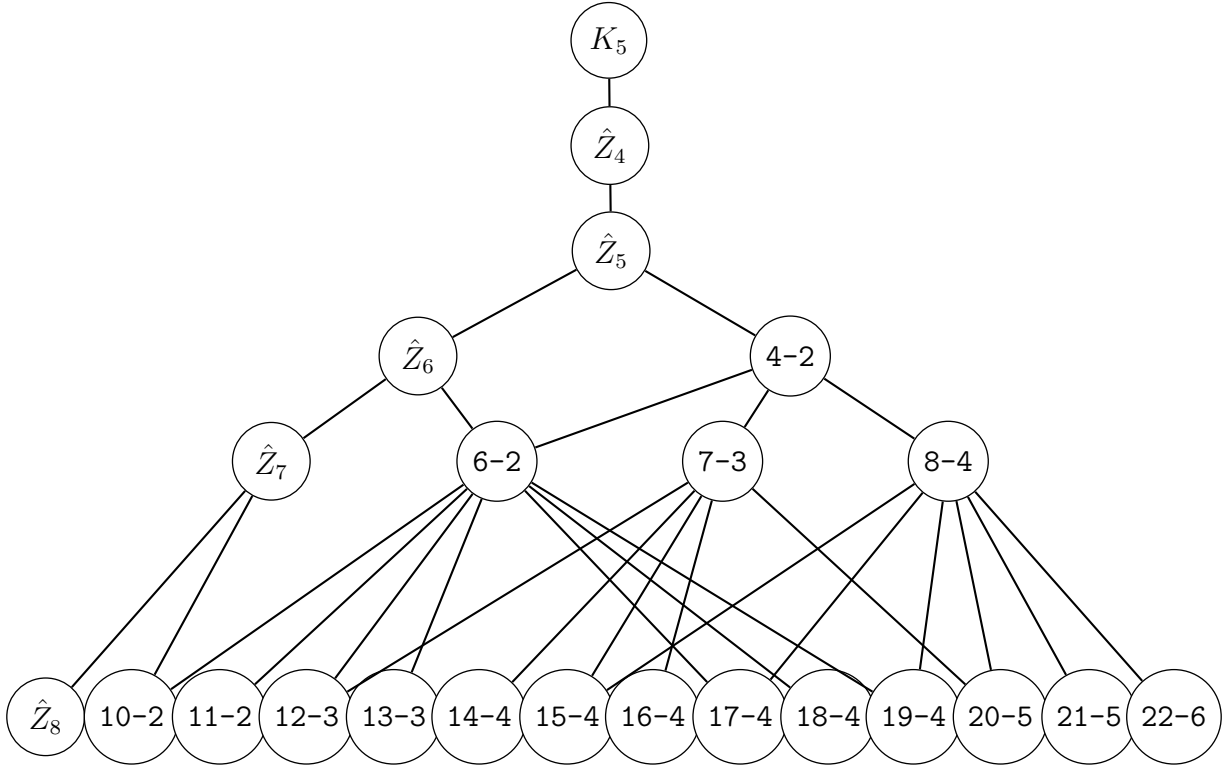
\begin{figure}
\begin{tikzpicture}
\GraphInit[vstyle=Normal]
\Vertex[L=\hbox{$K_5$},x=7.4878cm,y=7cm]{v0}
\Vertex[L=\hbox{$\hat{Z}_4$},x=7.5013cm,y=5.6cm]{v1}
\Vertex[L=\hbox{$\hat{Z}_5$},x=7.5175cm,y=4.2cm]{v2}
\Vertex[L=\hbox{$\hat{Z}_6$},x=\maxwidth*\ppbmult,y=2.8cm]{v3}
\Vertex[L=\hbox{$\hat{Z}_7$},x=\maxwidth*\prebottommult,y=1.4cm]{v4}
\Vertex[L=\hbox{$\hat{Z}_8$},x=0.0cm,y=\bottomheight]{v5}
\Vertex[L=\hbox{$\text{\texttt{10{-}2}}$},x=1.1422cm,y=\bottomheight]{v6}
\Vertex[L=\hbox{$\text{\texttt{11{-}2}}$},x=2.3136cm,y=\bottomheight]{v7}
\Vertex[L=\hbox{$\text{\texttt{12{-}3}}$},x=3.4669cm,y=\bottomheight]{v8}
\Vertex[L=\hbox{$\text{\texttt{13{-}3}}$},x=4.6268cm,y=\bottomheight]{v9}
\Vertex[L=\hbox{$\text{\texttt{14{-}4}}$},x=5.7638cm,y=\bottomheight]{v10}
\Vertex[L=\hbox{$\text{\texttt{15{-}4}}$},x=6.932cm,y=\bottomheight]{v11}
\Vertex[L=\hbox{$\text{\texttt{16{-}4}}$},x=8.0719cm,y=\bottomheight]{v12}
\Vertex[L=\hbox{$\text{\texttt{17{-}4}}$},x=9.2339cm,y=\bottomheight]{v13}
\Vertex[L=\hbox{$\text{\texttt{18{-}4}}$},x=10.402cm,y=\bottomheight]{v14}
\Vertex[L=\hbox{$\text{\texttt{19{-}4}}$},x=11.5365cm,y=\bottomheight]{v15}
\Vertex[L=\hbox{$\text{\texttt{20{-}5}}$},x=12.695cm,y=\bottomheight]{v16}
\Vertex[L=\hbox{$\text{\texttt{21{-}5}}$},x=13.8729cm,y=\bottomheight]{v17}
\Vertex[L=\hbox{$\text{\texttt{22{-}6}}$},x=15.0cm,y=\bottomheight]{v18}
\Vertex[L=\hbox{$\text{\texttt{4{-}2}}$},x=2*\maxwidth*\ppbmult,y=2.8cm]{v19}
\Vertex[L=\hbox{$\text{\texttt{6{-}2}}$},x=2*\maxwidth*\prebottommult,y=1.4cm]{v20}
\Vertex[L=\hbox{$\text{\texttt{7{-}3}}$},x=3*\maxwidth*\prebottommult,y=1.4cm]{v21}
\Vertex[L=\hbox{$\text{\texttt{8{-}4}}$},x=4*\maxwidth*\prebottommult,y=1.4cm]{v22}
\Edge[](v0)(v1)
\Edge[](v1)(v2)
\Edge[](v2)(v3)
\Edge[](v2)(v19)
\Edge[](v3)(v4)
\Edge[](v3)(v20)
\Edge[](v4)(v5)
\Edge[](v4)(v6)
\Edge[](v19)(v20)
\Edge[](v19)(v21)
\Edge[](v19)(v22)
\Edge[](v20)(v6)
\Edge[](v20)(v7)
\Edge[](v20)(v8)
\Edge[](v20)(v9)
\Edge[](v20)(v13)
\Edge[](v20)(v14)
\Edge[](v20)(v15)
\Edge[](v21)(v8)
\Edge[](v21)(v10)
\Edge[](v21)(v11)
\Edge[](v21)(v12)
\Edge[](v21)(v16)
\Edge[](v22)(v11)
\Edge[](v22)(v13)
\Edge[](v22)(v15)
\Edge[](v22)(v16)
\Edge[](v22)(v17)
\Edge[](v22)(v18)

\end{tikzpicture}
\caption[$K_5$ Family Tree]{$K_5$ descendants of up to order $10$. Two graphs are adjacent if and only if one can be obtained from the other by a double triangle expansion. The graphs are sorted by order from top to bottom, by level (see Section~\ref{sec:enumerative}) from left to right, and are labelled as given by the programs described in \cite{LaradjiDTDoK5} and which can be obtained from \cite{code}. The node labelling indicates a global identifier and the level of the graph: [global index]-[level].}
\label{fig:K5familytree}
\end{figure}

\subsection{Quantum field theoretic motivation}

To understand the interest of $K_5$ descendants for quantum field theory, first consider any connected $4$-regular graph $G$.  If we remove any one vertex of $G$ then we can view the result as a 4-point Feynman diagram in $\phi^4$ theory and we call such a graph a \emph{decompletion} of $G$.  Different decompletions will typically be non-isomorphic but $G$ can be uniquely reconstructed from any of its decompletions.  We will say that $G$ is the \emph{completion} of any of its decompletions.  

For those not familiar with Feynman diagrams, briefly: The edges of the graphs represent particles; the vertices particle interactions.  This particular quantum field theory has only a quartic interaction (that is the $4$ in $\phi^4$), and so all vertices must have degree 4.  When decompleting, we take the removal of the vertex to leave dangling ends of the edges incident to that vertex.  These dangling ends are called external edges and represent particles entering or exiting the system.  The purpose of Feynman diagrams is that they index certain integrals known as Feynman integrals, and summing all Feynman integrals corresponding to Feynman diagrams with some fixed external edges calculates the amplitude of the process with those external edges.  For more on perturbative quantum field theory see \cite{iz}.

Feynman integrals are complicated for many different reasons.  The integrals are typically divergent, leading to the need for renormalization (which we will not discuss further here, see \cite{ck0, patras} for some algebraic perspectives).  Even with this sorted out they are often difficult to compute, even numerically.  Their dependence on physical parameters such as masses of the particles and the momenta of the incoming and outgoing particles are complicated.  The kinds of functions you need to represent them are number theoretically sophisticated, beginning with multiple polylogarithms and moving into elliptic generalizations of polylogarithms and ultimately surely beyond.

We can ignore the dependence on parameters and issues of renormalization by setting all parameters to zero and taking the residue at the divergence.  The resulting integral is called the \emph{Feynman period} \cite{Sphi4} and is a nice mathematical object.  Even with all these sources of complexity removed the Feynman period is still a rich and interesting object.  It is defined as follows:

Let $H$ be a decompletion of $G$.  Assign a variable $a_e$ to each edge $e \in E(H)$ and define the (dual) \emph{Kirchhoff polynomial} or first Symanzik polynomial to be
\[
\Psi_H = \sum_{T} \prod_{e\not\in T}a_e,
\]
where the sum runs over all spanning trees of $H$.  For example, the Kirchhoff polynomial of a 3-cycle with edge variables $a_1$, $a_2$, $a_3$ is $a_1+a_2+a_3$ since removing any one edge of a cycle gives a spanning tree of the cycle.

Then the \emph{Feynman period} is
\[
\int_{a_e\geq 0} \frac{\Omega}{\Psi_H^2},
\]
where $\Omega = \sum_{i=1}^{|E(H)|}(-1)^ida_1\cdots da_{i-1}da_{i+1}\cdots da_{|E(H)|}$.  This integral converges provided $G$ is at least \emph{internally 6-edge connected}, which means that any way of removing fewer than 6 edges from $G$ either leaves $G$ connected or breaks $G$ into two components one of which is a single vertex. This corresponds to $H$ as a Feynman diagram having no subdivergences.

The Feynman period is a sensible algebro-geometric, or even motivic, object \cite{bek,BrSinform,Brbig,Brcosmic,Mmotives,Sphi4} which still captures some of the richness of Feynman integrals.  The numbers obtained from these integrals are number theoretically interesting, see \cite{bek,bkphi4,BrSinform,Sphi4,Snumbers}.

Additionally, two graphs with the same completion have the same Feynman period, see \cite{Sphi4}, and so while the actual integral of the Feynman period is defined in terms of a decompleted graph it makes sense to think of the Feynman period as an invariant of the completed graph.

\subsection{Kontsevich's conjecture}

From an algebro-geometric perspective, what numbers come out of a Feynman period is controlled by the geometry of the algebraic variety defined by the vanishing of the Kirchhoff polynomial.  Another way to access features of this geometry is by looking at the Kirchhoff polynomial over finite fields.  If the Feynman period is nice then the point counts over the finite fields $\mathbb{F}_p$ as a function of $p$ should be nice and vice versa.
More specifically, for a prime $p$ we will use the notation $[\Psi_H]_p$ for the number of points in the affine variety of $\Psi_H$ over $\mathbb{F}_p$.  
Inspired by known Feynman periods at the time being multiple zeta values, Kontsevich informally conjectured that $[\Psi_H]_p$ should be a polynomial in $p$.  This turned out to be very false \cite{BrBe}.  However, how badly Kontsevich's conjecture fails for a given graph turns out to be very insightful.

Thus, in order to better understand the Feynman period, Schnetz \cite{SFq} introduced the $c_2$ invariant.

\begin{definition}
  Let $H$ be a connected graph with at least 3 vertices, then the \emph{$c_2$-invariant} of $H$ at $p$ is
  \[
  c_2^{(p)}(H) = \frac{[\Psi_H]_p}{p^2} \mod p.
  \]
\end{definition}
That this is well defined is proved in \cite{SFq}.  If Kontsevich's conjecture were true for $H$ then $c_2^{(p)}(H)$ would be a constant, independent of $p$, and in particular would be the quadratic coefficient of the point count polynomial.

The $c_2$ invariant has interesting properties \cite{BrS,BrSY,Dc2,D4face}, yields interesting sequences in $p$ such as coefficient sequences of modular forms \cite{BrS3,Lmod}, and predicts properties of the Feynman period \cite{BrS,BrSY,SFq}.  The $c_2$ invariant is conjectured to be the same for graphs with the same completion \cite{BrS}.

\medskip

Returning now to the $K_5$ descendants, every $K_5$ descendant is $4$-regular and internally 6-edge-connected (discussed further on page \pageref{defn:completedprimitive}), so we can ask about the Feynman periods and the $c_2$ invariants of their decompletions.\\

Every decompletion of $K_5$ is $K_4$ which can be computed directly to have $c_2^{(p)}(K_4) = -1$ for all $p$.  The $c_2$ invariant is preserved by double triangle expansion \cite{BrS}, even when the double triangle in the completed graph involves the vertex that is removed upon decompletion \cite{FInvestc2}. Therefore, every decompletion of a $K_5$ descendant has $c_2^{(p)} = -1$ for all $p$.


More surprisingly, Brown and Schnetz, based on exhaustive calculation up to 10 loops\footnote{The \emph{loop number} of a graph is the dimension of its cycle space.}, conjectured \cite[Conjecture 25]{BrS3} that the $c_2$-invariant of a 4-point graph is $-1$ if and only if the graph is a decompletion of a $K_5$ descendant.  Furthermore, they observed that the only constant $c_2$-invariants which appear in their data are $0$ and $-1$, where $0$ corresponds to a drop in transcendental weight.  Note that with any finite number of primes, we cannot rule out large constant values, as by the remainder theorem any finite set of residues at different primes can arise from some fixed integer.  However the absence of any other small constant values in the data is suggestive that it may be the case that no constants other than $0$ and $-1$ are possible.  This is indicative of an overall sparsity of sequences which appear; as a further example, Brown and Schnetz call a sequence quasi-constant if it is constant for almost every prime or after a finite field extension and they only find three other quasi-constant sequences. On the other hand, going to arbitrary high loop order, sparsity does not seem evident \cite{Yprefix}.  This interplay between loop order and sequences appearing as $c_2$ invariants is likely to be important.

Returning to the concerns of the present paper, the computed data on $c_2$ invariants supports that the $K_5$ descendants form a very special class with this particular constant $c_2$.

The best hope for a resolution of Brown and Schnetz' conjecture is a better structural understanding of $K_5$ descendants, as the most likely way forward on the conjecture would be theorems saying that $c_2$ equal to $-1$ implies something on the structure of the graph.  This is likely to be very difficult as so far we only have results in the other direction and it is not clear what levers we could apply.  A good structural understanding of the $K_5$ descendants would give a strong hint of what such a theorem could look like and would be also be necessary for a proof of the conjecture following this plan.

Enumerative and asymptotic results on $K_5$ descendants are also interesting in view of Brown and Schnetz' conjecture as they help us understand the number of $K_5$ descendants relative to all completed $\phi^4$ graphs or other special classes of completed $\phi^4$ graphs.  Additionally the form of the generating functions tells us something more indirect about the structure of these graphs; we get rational generating functions for the special cases we can enumerate below showing that these classes are relatively tame.

More broadly we can also simply view Brown and Schnetz' conjecture as evidence that the class of $K_5$ descendants is a rich and interesting class worthy of study for its own sake.

\subsection{Article organization}

In Section~\ref{sec:prelim}, we will define what we need for the double triangle operations and summarize what is already known about the properties of double triangle transformations.  Section~\ref{sec:prelim} describes an important structural element: zigzag graphs.  Section~\ref{sec:enumerative} provides enumeration formulas for $K_5$ descendants with at most 4 more vertices than triangles.  The main results of that section are Propositions~\ref{prop:level0}, \ref{prop:level1}, \ref{prop:level2}, \ref{prop:level3}, and \ref{prop:level4}, enumerating $K_5$ descendants with $L$ more vertices than triangles for $L=0,1,2,3,4$ respectively.  The process whereby these graph classes were enumerated is automatable.  

Section~\ref{sec:results} contains the proof that the minimum number of triangles a $K_5$ descendant can have is 4 (the naive bound would be only 2).  This is done by first introducing the chain vector, an encoding of the lengths of the maximal zigzag subgraphs, then by considering how triangles in the neighbourhood of a double triangle operation can be changed, and then finally using all this to prove the result. 

\subsection{Acknowledgements} We are grateful to various institutions and funding bodies that have provided support for this work: The Laboratoire Bordelaise en Recherches Informatiques (LaBRI) at Universit\'e de Bordeaux,  where a significant part of the work was completed; we thank Adrian Tanasa in particular for his help;  KY is supported by a Humboldt fellowship from the Alexander von Humboldt foundation, in addition to NSERC and the Canada Research Chair program; MM is partially supported by NSERC Discovery Grant;  ML was supported by NSERC and Simon Fraser University.

\section{Double triangle operations and zigzag graphs}
\label{sec:prelim}
\subsection{Double triangle operations}

In this section, double triangle reduction (DTR) (Figure~\ref{fig:dtr}) and its inverse operation, double triangle expansion (DTE) (Definition~\ref{defn:dte} and Figure~\ref{fig:DTEtriangleandanedge}), are discussed. A \emph{double triangle} is the complete tripartite graph $K_{2,1,1}$, and a \emph{triple triangle} is $K_{3,1,1}$. A double triangle in a graph $G$ is \emph{proper} if it is not a subgraph of a triple triangle.


In \emph{double triangle reduction} of a graph $G$, a proper double triangle $(v_1,v_2,v_3,v_4)$, with $(v_1,v_2,v_3)$ and $(v_2,v_3,v_4)$ triangles, is turned into a triangle $(v_1, v_2, v_4)$ by removing the edges $(v_1, v_3), (v_2, v_3)$, and $(v_4, v_3)$, identifying the vertices $v_2$ and $v_3$, and adding the edge $(v_1,v_4)$ \cite[Definition 2.18]{Sphi4}\label{defn:dtr}. Clearly, the resulting graph from a DTR depends on the choice of double triangle.

In this paper, we restrict DTR to proper double triangles in order to simplify the statements of results like Theorem~\ref{thm:ancestor} and to be consistent with other sources such as \cite{Sphi4}.


\begin{figure}
	\newcommand{\belowyshift}{-0.3cm}
	\begin{subfigure}{\textwidth}
		\centering
		\begin{tikzpicture}[every node/.style={draw,circle,very thick, fill=black, minimum size=4pt, inner sep=0pt}]
		\node[label={$v_2$}] (v2) at (0,2) {};
		\node[label={[below,yshift=\belowyshift]$v_1$}] (v1) at (0,0) {};
		\node[label={[below,yshift=\belowyshift]$v_4$}] (v3) at (2,0) {};
		\node[label={$v_3$}] (v4) at (2,2) {};
		\draw
		(v4)--(v1)--(v2)--(v4)
		(v4)--(v3)--(v2)
		(-0.5,2.5)--(v2)
		(-0.5,0)--(v1)--(-0.5,0.5)
		(2.5,0.5)--(v3)--(2.5,0)
		(v4)--(2.5,2.5)
		;
		\end{tikzpicture}
		\begin{tikzpicture}[scale=0.9]
		\node (v1) at (0,0) {};
		\node (v2) at (1,2) {};
		\draw[->]
		(0.25,2)--(0.75,2);
		\end{tikzpicture}
		\begin{tikzpicture}[every node/.style={draw,circle,very thick, fill=black, minimum size=4pt, inner sep=0pt}]
		\node[label={$v_2$}] (v2) at (0,2) {};
		\node[label={[below,yshift=\belowyshift]$v_1$}] (v1) at (0,0) {};
		\node[label={[below,yshift=\belowyshift]$v_4$}] (v3) at (2,0) {};
		\node[label={$v_3$}] (v4) at (2,2) {};
		\draw
		(v1)--(v2)
		(v3)--(v2)
		(-0.5,2.5)--(v2)
		(-0.5,0)--(v1)--(-0.5,0.5)
		(2.5,0.5)--(v3)--(2.5,0)
		(v4)--(2.5,2.5)
		;
		\end{tikzpicture}
		\begin{tikzpicture}[scale=0.9]
		\node (v1) at (0,0) {};
		\node (v2) at (1,2) {};
		\draw[->]
		(0.25,2)--(0.75,2);
		\end{tikzpicture}
		\begin{tikzpicture}[every node/.style={draw,circle,very thick, fill=black, minimum size=4pt, inner sep=0pt}]
		\node[label={$v_5$}] (v5) at (1,2) {};
		\node[label={[below,yshift=\belowyshift]$v_1$}] (v1) at (0,0) {};
		\node[label={[below,yshift=\belowyshift]$v_4$}] (v4) at (2,0) {};
		\draw
		(v1)--(v5)
		(v4)--(v5)
		(0.5,2.5)--(v5)--(1.5,2.5)
		(-0.5,0)--(v1)--(-0.5,0.5)
		(2.5,0.5)--(v4)--(2.5,0)
		;
		\end{tikzpicture}
		\begin{tikzpicture}[scale=0.9]
		\node (v1) at (0,0) {};
		\node (v2) at (1,2) {};
		\draw[->]
		(0.25,2)--(0.75,2);
		\end{tikzpicture}
		\begin{tikzpicture}[every node/.style={draw,circle,very thick, fill=black, minimum size=4pt, inner sep=0pt}]
		\node[label={$v_5$}] (v5) at (1,2) {};
		\node[label={[below,yshift=\belowyshift]$v_1$}] (v1) at (0,0) {};
		\node[label={[below,yshift=\belowyshift]$v_4$}] (v4) at (2,0) {};
		\draw
		(v1)--(v4)
		(v1)--(v5)
		(v4)--(v5)
		(0.5,2.5)--(v5)--(1.5,2.5)
		(-0.5,0)--(v1)--(-0.5,0.5)
		(2.5,0.5)--(v4)--(2.5,0)
		;
		\end{tikzpicture}
	\end{subfigure}
	\caption[Double triangle reduction]{The operation $\mbox{DTR}((v_1,v_2,v_3,v_4))$ is shown here. In the intermediate step, vertices $v_2,v_3$ are identified to create $v_5$.}
	\label{fig:dtr}
\end{figure}
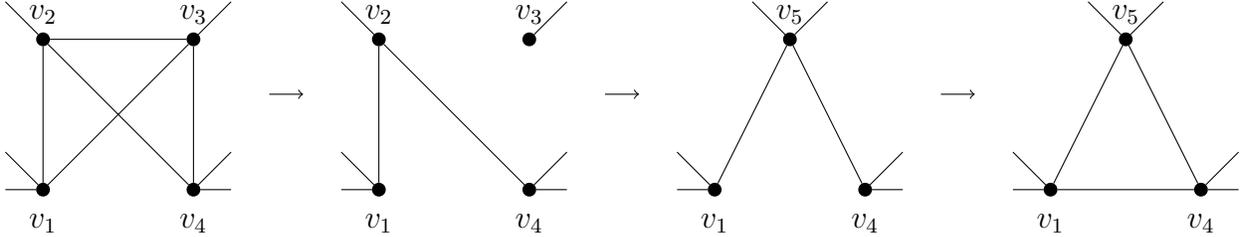


 
If $H$ is graph $G$ after one DTE, we call $H$ a \emph{child} of $G$, and $G$ a \emph{parent} of $H$. The child of a graph $G$ depends on the choice of triangle $T$ and the choice of a vertex $v$ in $T$.  A priori the child also depends on the choice of neighbour of $v$ but since we are working with unlabelled graphs the choice of neighbour is irrelevant as the two possibilities give isomorphic results.

%

One of the earliest occurrences of double triangle expansion is in a 2011 paper by Brown and Yeats \cite{BrY}, in which they proved a result on linear reducibility and weight drop. A Feynman graph is \emph{linearly reducible} if its Feynman integral is linearly reducible. Roughly speaking, a Feynman integral is \emph{linearly reducible} if, for some ordering of the integration variables, it can be integrated iteratively using multiple polylogarithms \cite{MYlinreduc}.
A graph is said to have \emph{weight drop} if the weight of its period is less than the maximal transcendental weight, which is equal to $2\ell-3$ for a $\phi^4$-graph with loop number $\ell$. The \emph{transcendental weight} of a number $\nu$ is very roughly the minimum number of nested integrals in an integral expression, with rational integrand and limits, that evaluates to $\nu$~\cite{BrY}.


It has also been shown that double triangle expansion preserves the $c_2$-invariant \cite[Corollary 34, p16]{BrS} and even does so for what would be a double triangle in the completed graphs but where the vertex removed on decompletion is involved in the double triangle \cite{FInvestc2}. 




DTR and DTE have other interesting properties. Schnetz proved that DTR preserves completed primitiveness \cite[Proposition 2.19, p20-1]{Sphi4}\label{prop:preservecompprim}, and is commutative in the sense that the order in which we do double triangle reductions does not matter (see Figure \ref{fig:dtrscommute} which is adapted from \cite{Sphi4}) if the graph is completed primitive. A graph $G$ is \emph{completed primitive} if it is internally $6$-edge-connected \cite[p10]{Sphi4}. A graph is \emph{internally $k$-edge-connected} if the only way to disconnect it by removing $k-1$ or fewer edges is to separate off a single vertex.  A Feynman graph $\gamma$ is \emph{divergent} if its Feynman integral is divergent, and is \emph{primitive divergent} if it is divergent and no proper subset of its integration variables is divergent \cite[p5-6]{Ycombpers}. Schnetz proved in \cite[p10]{Sphi4} that $G$ is completed primitive if and only if some decompletion $\gamma$ is primitive divergent which occurs if and only if every decompletion is primitive divergent.\label{defn:completedprimitive}\\

DTRs also commute with an operation called the product split, which implies that any sequence of DTRs and product splits terminate at a unique \emph{ancestor}. The product split
is defined as follows. Suppose a completed primitive graph $G$ can be obtained by identifying two triangles from two completed primitive graphs $G_1,G_2$, respectively, and removing the triangle edges. Then, either of $G_1,G_2$ is a \emph{product split} of $G$, and $G$ is a \emph{product} of $G_1,G_2$. \label{defn:productsplit} Any $3$-connected completed primitive graph $G$ is a product of two completed primitive graphs \cite[Theorem 2.10, p14]{Sphi4}.

\begin{figure}
	\centering
	\begin{tikzpicture}[scale=.5, every node/.style={draw,shape=circle,fill=black,scale=0.4}]

	\node (v1) at (0,0) {};
	\node (v2) at (0,2) {};
	\node (v3) at (2,0) {};
	\node (v4) at (2,1) {};
	\node (v5) at (2,2) {};
	\draw
	(v1)--(v2)--(v3)--(v4)--(v5)--(v1)
	(v4)--(v2)--(v5)
	(v4)--(v1)--(v3)
	[-] (v5)  to [out=-30,in=90,in looseness=0.5] (2.5cm,1cm) to [out=-90,in=30,out looseness=.5](v3)
	;
	\end{tikzpicture}
	\hspace{0.2cm}
	\begin{tikzpicture}[scale=.5]
	\node at (0,2) {};
	\node at (0,1) {\huge{$\times$}};
	\end{tikzpicture}
	\hspace{0.2cm}
	\begin{tikzpicture}[scale=.5, every node/.style={draw,shape=circle,fill=black,scale=0.4}]

	\node (v1) at (0,0) {};
	\node (v2) at (0,2) {};
	\node (v3) at (-2,0) {};
	\node (v4) at (-2,1) {};
	\node (v5) at (-2,2) {};
	\draw
	(v1)--(v2)--(v3)--(v4)--(v5)--(v1)
	(v4)--(v2)--(v5)
	(v4)--(v1)--(v3)
	[-] (v5)  to [out=-150,in=90,in looseness=0.5] (-2.5cm,1cm) to [out=-90,in=150,out looseness=.5](v3)
	;
	\end{tikzpicture}
	\hspace{0.5 cm}
	\begin{tikzpicture}[scale=.5]
	\node (v1) at (0,2) {};\node  at (0,0) {};
	\node at (0,1) {$=$};
	\end{tikzpicture}
	\hspace{0.5 cm}
	\begin{tikzpicture}[scale=.5, every node/.style={draw,shape=circle,fill=black,scale=0.4}]

	\node (v1) at (0,0) {};
	\node (v2) at (0,2) {};
	\node (v3) at (2,0) {};
	\node (v4) at (2,1) {};
	\node (v5) at (2,2) {};
	\node (v6) at (4,0) {};
	\node (v7) at (4,2) {};
	\draw
	(v1)--(v2)--(v3)
	(v5)--(v1)
	(v4)--(v2)--(v5)
	(v4)--(v1)--(v3)
	(v6)--(v7)
	(v4)--(v6)--(v3)
	(v6)--(v5)--(v7)
	(v4)--(v7)--(v3)
	;
	\end{tikzpicture}
	\caption{The product joins two graphs each with a distinguished triangle by identifying the triangles and then removing the identified edges.  Here it is illustrated with $K_5\times K_5$.  The inverse operation is the product split.  The product split is important for defining the ancestor of a graph.}
	\label{fig:K5timesK5}
\end{figure}
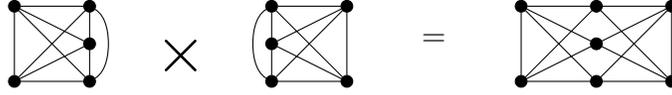

A completed primitive graph is \emph{reducible} if it has vertex-connectivity 3, and is \emph{irreducible} otherwise \cite[Definition 2.9, p14]{Sphi4}. If $G$ is a reducible completed primitive graph, then $G$ is the product of two completed primitive graphs $G_1,G_2$. Furthermore, the period of $G$ is the product of the periods of $G_1,G_2$ \cite[Theorem 2.10, p14]{Sphi4}. An example of a product split is shown in Figure \ref{fig:K5timesK5}.

\begin{thm}[Schnetz. {\cite[Definition 2.23, p22]{Sphi4}}]
	Let $G$ be a completed primitive graph. DTRs and product splits of $G$ commute, and any sequence of DTRs and product splits terminates at a product of double-triangle-free irreducible graphs.
	\label{thm:ancestor}
\end{thm}

\begin{figure}
	\centering
	\begin{tikzpicture}[every node/.style={draw,thick,circle,fill,inner sep=0.8pt}]
	\node (v0) at (0,0) {};
	\node (v1) at (-1,0) {};
	\node (v2) at (0,1) {};
	\node (v3) at (1,0) {};
	\node (v4) at (0,-1) {};
	\node (v5) at (1.5,0) {};
	\draw 
	(v0)--(v1)--(v2)--(v3)--(v4)--(v0)
	(v5)--(v3)--(v0)--(v2)
	(-1.3,-0.3)--(v1)--(-1.3,0.3)
	(v2)--(0,1.4)
	(-0.3,-1.3)--(v4)--(0.3,-1.3)
	(1.8,-0.3)--(v5)--(1.8,0.3)
	(v5)--(1.8,0)
	;
	\end{tikzpicture}
	\hspace{0.2cm}
	\begin{tikzpicture}
\node (v1) at (0,-1) {};
\draw[->]
(0,0)--(0.5,0);
\end{tikzpicture}
\hspace{0.2cm}
	\begin{tikzpicture}[every node/.style={draw,thick,circle,fill,inner sep=0.8pt}]
	\node (v1) at (-1,0) {};
	\node (v2) at (0,1) {};
	\node (v3) at (1,0) {};
	\node (v4) at (0,-1) {};
	\node (v5) at (1.5,0) {};
	\draw 
(v1)--(v2)--(v3)--(v4)--(v2)
(v5)--(v3)--(v1)
(-1.3,-0.3)--(v1)--(-1.3,0.3)
(v2)--(0,1.4)
(-0.3,-1.3)--(v4)--(0.3,-1.3)
(1.8,-0.3)--(v5)--(1.8,0.3)
(v5)--(1.8,0)
;
\end{tikzpicture}
\hspace{0.2cm}
	\begin{tikzpicture}
	\node (v1) at (0,-1) {};
\draw[->]
(0,0)--(0.5,0);
\end{tikzpicture}
\hspace{0.2cm}
	\begin{tikzpicture}[every node/.style={draw,thick,circle,fill,inner sep=0.8pt}]
\node (v1) at (-1,0) {};
\node (v2) at (0,1) {};
\node (v4) at (0,-1) {};
\node (v5) at (1.5,0) {};
\draw 
(v1)--(v2)--(v4)--(v1)
(v5)--(v2)
(-1.3,-0.3)--(v1)--(-1.3,0.3)
(v2)--(0,1.4)
(-0.3,-1.3)--(v4)--(0.3,-1.3)
(1.8,-0.3)--(v5)--(1.8,0.3)
(v5)--(1.8,0)
;
\end{tikzpicture}

	\caption{This figure shows $2$ DTR's done on two double triangles that share a triangle.  Notice that the order of the DTR's doesn't matter in this case; the middle graph is identical no matter which DTR is done first.  Even leaving the 4-regular situation, any additional edges adjacent to the middle vertex in the first graph would be adjacent to either the top or mid-right vertices in the second graph, depending on which DTR is done first, and then would in both cases be adjacent to the top vertex at the end, showing the DTRs commuting also in this case. In general, DTR's commute in the sense that if a completed primitive graph $G$ has distinct double triangles $D_1$ and $D_2$, then DTR of $D_1$ and $D_2$ in either order results in the same graph, see \cite{Sphi4}.}
\label{fig:dtrscommute}
\end{figure}
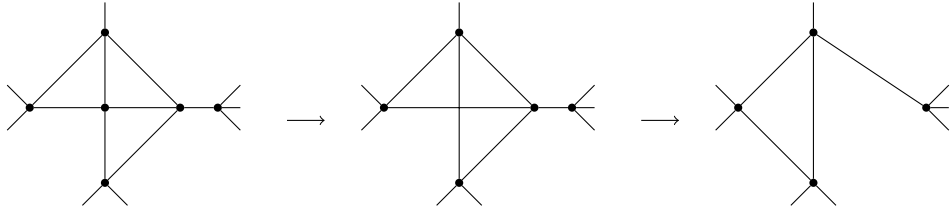

Let $A$ be the graph(s) that is the termination of a sequence of DTRs and product splits on a graph $G$ as in Theorem \ref{thm:ancestor}. We call $G$ a \emph{descendant} of $A$, and $A$ the \emph{ancestor} of $G$. The set of descendants of the ancestor $A$ is called the \emph{family}\footnote{We take the convention that any graph is a descendant of itself and is thus included in its family.} of $A$. To find the ancestor of a family, one could start with any member of that family, and perform DTRs and product splits until unable to do so. The remaining graph(s) is the ancestor of the family. Note that the ancestor is not necessarily one graph, but it is, in the most general case, a multiset. For instance, the ancestor of $K_5\times K_5$, shown in Figure \ref{fig:K5timesK5}, is the multiset $\{K_5,K_5\}$.\\

Since $K_5$ has no proper double triangles, no further DTRs can be done. Since it has no non-trivial $3$-vertex cut, it cannot be (non-trivially) product split. Thus, $K_5$ is the ancestor of its family. The $c_2$-invariant of any decompletion of $K_5$ is $c_2(K_4) = -1$ (for all primes) (this is a special case of the $c_2$ calculation for zigzags, see \cite{BrS}) so by the fact that DTE preserves the $c_2$ invariant, we have that for any decompletion $g$ of any descendant $G$ of $K_5$, $c_2(g) = c_2(K_4) = -1$.\\

\subsection{Zigzags}

One of the basic building blocks of $K_5$ descendants is the zigzag $Z_n^*$ (Figure \ref{fig:zigzags}), defined as the undirected graph with vertex set $\{1,2,\cdots,n+2\}$ and edge set 
\begin{align*}
\mbox{E}(Z_n):=\{(i,i+1): 1 \leq i \leq n+1\}\cup\{(i,i+2): 1\leq i \leq n\}.
\end{align*}
Adding additional edges to $Z_n^*$ to make it $4$-regular (there is a unique way to do so) results in the \emph{$1$-zigzag graph} $\hat{Z_n}$ (Figure \ref{fig:1zigzags}). Equivalently, $\hat{Z_n}$ can be defined as the circulant graph\footnote{The \emph{circulant graph} $C_n(i_1,i_2, \cdots, i_k)$ is the graph with vertex set $\{1,2,\cdots,n\}$ such that vertices $i$ and $j$ are adjacent if and only if $i-j\equiv\pm i_l \mod{n}$ for some $1\leq l \leq k$. \cite[Definition 1.1]{Ycirc}} $C_{n+2}(1,2)$ \cite{Ycirc}. The \emph{decompleted $1$-zigzag graph} $Z_n$ (Figure \ref{fig:decompletedzigzags}) is the decompletion of $\hat{Z_n}$, which is unique since circulant graphs are vertex-transitive.

\begin{figure}
\begin{subfigure}{.45\textwidth}
\center
	\begin{tikzpicture}[scale=.5,every node/.style={draw,shape=circle,fill=black,scale=0.4}]
	\node (v1) at (0,0) {};
	\node (v2) at (1,2) {};
	\node (v3) at (2,0) {};
	\node (v4) at (3,2) {};
	\node (v5) at (4,0) {};
	\draw
	(v1)--(v2)--(v3)--(v4)--(v5)
	(v2)--(v4)
	(v1)--(v3)--(v5);
	\end{tikzpicture}
	\hspace{1cm}
	\begin{tikzpicture}[scale=.5,every node/.style={draw,shape=circle,fill=black,scale=0.4}]
	\node (v1) at (0,0) {};
	\node (v2) at (1,2) {};
	\node (v3) at (2,0) {};
	\node (v4) at (3,2) {};
	\node (v5) at (4,0) {};
	\node (v6) at (5,2) {};
	\draw
	(v1)--(v2)--(v3)--(v4)--(v5)--(v6)
	(v2)--(v4)--(v6)
	(v1)--(v3)--(v5);
	\end{tikzpicture}	
	\caption{Zigzag graphs, $Z_3^*$ and $Z_4^*$}
\end{subfigure}
\begin{subfigure}{.5\textwidth}

	\centering
	\begin{tikzpicture}[scale=.5, every node/.style={draw,shape=circle,fill=black,scale=0.4}]

	\node (v2) at (1,2) {};
	\node (v3) at (2,0) {};
	\node (v4) at (3,2) {};
	\node (v5) at (4,0) {};
	\node (v6) at (5,2) {};
	\draw
	(v2)--(v3)--(v4)--(v5)--(v6)
	(v2)--(v4)--(v6)
	(v3)--(v5)
	[-] (v2)  to [out=30,in=180,in looseness=0.5] (3cm,2.5cm) to [out=0,in=150,out looseness=1](v6);
	\end{tikzpicture}
	\hspace{.5cm}	
	\begin{tikzpicture}[scale=.5, every node/.style={draw,shape=circle,fill=black,scale=0.4}]

	\node (v2) at (1,2) {};
	\node (v3) at (2,0) {};
	\node (v4) at (3,2) {};
	\node (v5) at (4,0) {};
	\node (v6) at (5,2) {};
	\node (v7) at (6,0) {};
	\draw
	(v2)--(v3)--(v4)--(v5)--(v6)--(v7)
	(v2)--(v4)--(v6)
	(v3)--(v5)--(v7)
	[-] (v2)  to [out=30,in=180,in looseness=0.5] +(3.5cm,0.7cm) to [out=0,in=75,out looseness=1](v7);
	\end{tikzpicture}
	
\caption{Decompleted 1-zigzag graphs, $Z_4$ and $Z_5$}
\vspace{5pt}
\end{subfigure}
\begin{subfigure}{.5\textwidth}
	\centering
	\begin{tikzpicture}[scale=.5, every node/.style={draw,shape=circle,fill=black,scale=0.4}]

	\node (v1) at (0,0) {};
	\node (v2) at (1,2) {};
	\node (v3) at (2,0) {};
	\node (v4) at (3,2) {};
	\node (v5) at (4,0) {};
	\node (v6) at (5,2) {};
	\draw
	(v1)--(v2)--(v3)--(v4)--(v5)--(v6)
	(v2)--(v4)--(v6)
	(v1)--(v3)--(v5)
	[-] (v1)  to [out=100,in=180,in looseness=1] (2cm,3cm) to [out=0,in=140,out looseness=1](v6)
	[-] (v2)  to [out=30,in=180,in looseness=0.5] (3cm,2.5cm) to [out=0,in=150,out looseness=1](v6)
	[-] (v1)  to [out=-30,in=180,in looseness=1] (2cm,-0.5cm) to [out=0,in=-150,out looseness=1](v5);
	\end{tikzpicture}
	\hspace{1cm}	
	\begin{tikzpicture}[scale=.5, every node/.style={draw,shape=circle,fill=black,scale=0.4}]

	\node (v1) at (0,0) {};
	\node (v2) at (1,2) {};
	\node (v3) at (2,0) {};
	\node (v4) at (3,2) {};
	\node (v5) at (4,0) {};
	\node (v6) at (5,2) {};
	\node (v7) at (6,0) {};
	\draw
	(v1)--(v2)--(v3)--(v4)--(v5)--(v6)--(v7)
	(v2)--(v4)--(v6)
	(v1)--(v3)--(v5)--(v7)
	[-] (v1)  to [out=117,in=180,in looseness=1.5] (2cm,3cm) to [in=150, out=0,out looseness=1.5](v6)
	[-] (v7)  to [out=63,in=0,in looseness=1.5] (4cm,3cm) to [in=30, out=180,out looseness=1.5](v2)
	[-] (v1)  to [out=-30,in=180,in looseness=0.5] (3cm,-0.7cm) to [out=0,in=-150,out looseness=0.5](v7);
	\end{tikzpicture}
	\caption{1-zigzag graphs, $\hat{Z_4}$ and $\hat{Z_5}$}
	
\end{subfigure}
\caption{Zigzag graphs and useful variants}
\label{fig:zigzags}
\label{fig:1zigzags}
\label{fig:decompletedzigzags}
\end{figure}
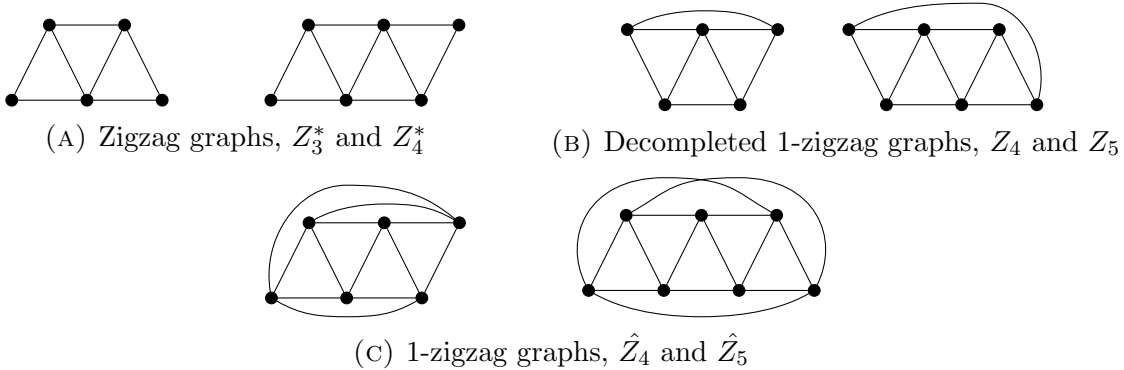

In the literature, the term zigzag usually refers to $Z_n$, such as in \cite{Ycirc, Szigzag, BrS}. We make the generalization here from zigzags to $n$-zigzag graphs (Figure \ref{fig:nzigzags}), as many $n$-zigzag graphs are $K_5$ descendants. In view of the generalization, we will refer to the zigzag graphs $Z_n$ as decompleted $1$-zigzag graphs and reserve the term zigzag alone to refer to $Z_n^*$.

\def\avar{1.5cm}

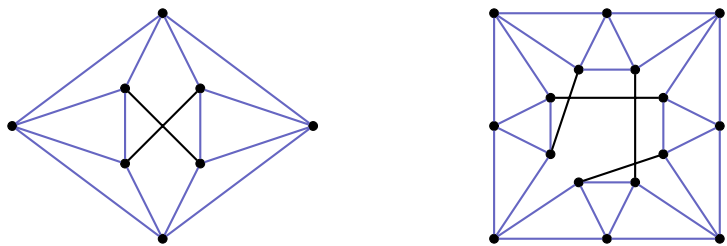
\begin{figure}
	\centering	
	\begin{tikzpicture}[rotate=90,every node/.style={draw,shape=circle,fill=black,scale=0.3}]
	\node (v1) at (0,0) {};
	\node (v2) at (1,-0.5) {};
	\node (v3) at (1.5,-2) {};
	\node (v4) at (2,-0.5) {};
	\node (v5) at (3,0) {};
	\node (v6) at (2,0.5) {};
	\node (v7) at (1.5,2) {};
	\node (v8) at (1,0.5) {};
	\draw[blue!60!black!60, thick]
	(v1)--(v2)--(v3)--(v4)--(v5)
	(v1)--(v3)--(v5)
	(v2)--(v4)
	(v5)--(v6)--(v7)--(v8)--(v1)
	(v5)--(v7)--(v1)
	(v6)--(v8);
	\draw[thick]
	(v4)--(v8)
	(v2)--(v6);
	\end{tikzpicture}
	\hspace{2cm}
	\begin{tikzpicture}[every node/.style={draw,shape=circle,fill=black,scale=0.3}]
	\node (1) at (0,0) {};
	\node (2) at (0.5*\avar,0.75*\avar) {};
	\node (3) at (0,\avar) {};
	\node (4) at (0.5*\avar,1.25*\avar) {};
	\node (5) at (0,2*\avar) {};
	\node (6) at (0.75*\avar,1.5*\avar) {};
	\node (7) at (\avar,2*\avar) {};
	\node (8) at (1.25*\avar,1.5*\avar) {};
	\node (9) at (2*\avar,2*\avar) {};
	\node (10) at (1.5*\avar,1.25*\avar) {};
	\node (11) at (2*\avar,1*\avar) {};
	\node (12) at (1.5*\avar,0.75*\avar) {};
	\node (13) at (2*\avar,0) {};
	\node (14) at (1.25*\avar,0.5*\avar) {};
	\node (15) at (\avar,0) {};
	\node (16) at (0.75*\avar,0.5*\avar) {};
	\draw[blue!60!black!60, thick]
	(1)--(2)--(3)--(4)--(5)--(6)--(7)--(8)--(9)--(10)--(11)--(12)--(13)--(14)--(15)--(16)--(1)
	(1)--(3)--(5)--(7)--(9)--(11)--(13)--(15)--(1)
	(2)--(4) (6)--(8) (10)--(12) (14)--(16);
	\draw[thick]
	(2)--(6) (4)--(10) (8)--(14) (12)--(16);
	\end{tikzpicture}
	\caption[n-zigzag graphs]{A 2-zigzag graph \emph{(left)} and a 4-zigzag graph \emph{(right)}.}
	\label{fig:nzigzags}
\end{figure}

The decompleted $1$-zigzag graphs $Z_n$ are a class of graphs that have been previously studied in the context of $\phi^4$-theory. In \cite{Szigzag}, Brown and Schnetz proved the closed form of the period of $Z_n$ for $n \geq 3$, originally conjectured by Broadhurst and Kreimer in \cite{bkphi4}. In \cite{BrS}, Brown and Schnetz proved that $c_2^{(p)}(Z_n)=-1$ for all primes $p$ and all $n \geq 3$. This makes them the only nontrivial infinite class for which the Feynman period is known \cite{Szigzag}, and one of the few infinite classes for which the $c_2$-invariant is known \cite{Ycirc,CYgrid}.

It is useful to name the different types of vertices in a zigzag $Z_n^*$. We call the degree $2$, $3$ and $4$ vertices of $Z_n^*$ the \emph{end}, \emph{chord}, and \emph{internal} vertices, respectively. A \emph{maximal} zigzag in a graph~$G$ is a zigzag subgraph $Z_n^*$ that is not a proper subgraph of any zigzag subgraph of~$G$. Maximal zigzags in $K_5$-descendants are either vertex-disjoint from other maximal zigzags or attached to other zigzags through their end vertices (see Lemma~\ref{lem zigzag shape} and Corollary~\ref{lem K5 desc is pseudo desc}). These groups of zigzags are called \emph{chains} (Figure~\ref{fig:chain332}), and are defined as follows.

\begin{figure}
	\centering	
	\begin{tikzpicture}[scale =.4]
	
	\coordinate (v1) at (0,0);
	\coordinate (v2) at (1,2);
	\coordinate (v3) at (2,0);
	\coordinate (v4) at (3,2);
	\coordinate (v5) at (4,0) ;
	\coordinate (v6) at (5,2);
	\coordinate (v7) at (6,0);
	\coordinate (v8) at (7,2);
	\coordinate (v9) at (8,0) ;
	\coordinate (v10) at (9,2);
	\coordinate (v11) at (10,0);
	\coordinate (v12) at (11,2);	
	
	\draw
	(v2)--(v3)--(v4)
	(v6)--(v7)--(v8)
	(v10)--(v12)
	(v1)--(v11)
	(v1)--(v2)--(v4)--(v5)--(v6)--(v8)--(v9)--(v10)--(v11)--(v12)
	;
	\foreach \x in  {1,2,...,12} {
	\filldraw(v\x) circle (6pt);
	}
	
	\end{tikzpicture}
	
	\caption[A $(3,3,2)$-chain]{A $(3,3,2)$-chain.}
	\label{fig:chain332}
\end{figure}
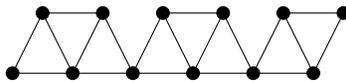

%

\begin{definition}[Chain]
	\label{defn:chain}
	Let $k, n_1, \cdots, n_k \in \mathbb{N}$. An open $(n_1, \cdots, n_k)$-chain is the graph formed by identifying an end vertex of $Z_{n_i}$ with an end vertex of $Z_{n_{i+1}}$, iteratively for each $i \in \{1, \cdots, k-1\}$. A closed $(n_1, \cdots, n_k)$-chain is the graph formed from an open $(n_1, \cdots, n_k)$-chain by identifying the remaining end vertex of $Z_{n_k}$ with the remaining end vertex of $Z_{n_1}$.  	
\end{definition}

For $n \geq 2$, an \emph{$n$-zigzag graph} is a closed $(k_1, \cdots, k_n)$-chain for some $k_1, \cdots, k_n \in \mathbb{N}$, with additional edges to make it $4$-regular\label{defn:nzigzag}. For a fixed $(k_1, \cdots, k_n)$, there are, in general, several $n$-zigzags, several of which might be descendants of $K_5$.

\section{Enumeration of $K_5$ descendants of small level}
\label{sec:enumerative}
Exhaustive enumeration of small $K_5$ descendants revealed apparent structure when the difference between the number of vertices and the number of triangles is fixed. Table~\ref{tbl:ordertri}, copied from~\cite{LaradjiDTDoK5} illustrates the sequences that motivate the following definition. We define the \emph{level} of a graph $G$, denoted $L(G)$, to be the number of vertices of $G$ minus the number of triangles of $G$.  
This turns out to be a reasonable approach, and in this section we describe how to 
derive explicit generating functions of the the sequences for $\mbox{L}(G)=0,1,2,3,4$. In each case they are rational functions with very simple denominators. The results are proved by a careful case analysis that we predict should be somewhat automatable towards the aim or determining generating functions of higher levels. We conjecture that for each level the generating function is a rational, with straightforward denominators. 

\begin{table}\center
	\caption{The number of non-isomorphic unlabelled $K_5$-descendants with $n$ vertices and exactly $t$ triangles. }
	\begin{tabular}{|c|c|c|c|c|c|c|c|c|c|c|c|}
			\hline
			$n\backslash t$ & {\bf 4} & \bf 5 & \bf 6 & \bf 7 & \bf 8 & \bf 9 & \bf 10 & \bf 11 & \bf 12 &\bf  13 & \bf 14 \\ \hline
			\bf 5 & \cellcolor{gray!10} 0 & 0 & 0 & 0 & 0 & 0 & 1 & 0 & 0 &0&0\\ \hline
			\bf 6 &  \cellcolor{gray!40}0 & \cellcolor{gray!10}0 &  0 & 0 & 1 & 0 & 0 & 0 & 0 &0&0\\ \hline
			\bf 7 & \cellcolor{gray!20}0 & \cellcolor{gray!40} 0 & \cellcolor{gray!10} 0 & \cellcolor{gray!50} 1 & 0 & 0 & 0 & 0 & 0 &0&0\\ \hline
			\bf 8 & \cellcolor{gray!60}0 & \cellcolor{gray!20} 0 & \cellcolor{gray!40}1 & \cellcolor{gray!10} 0 & \cellcolor{gray!50}1 & 0 & 0 & 0 & 0 &0&0\\ \hline
			\bf 9 & 0 & \cellcolor{gray!60}1 & \cellcolor{gray!20} 1 & \cellcolor{gray!40} 1 & \cellcolor{gray!10} 0 & \cellcolor{gray!50}1 & 0 & 0 & 0 &0&0\\ \hline
			\bf 10 & 1 & 2 & \cellcolor{gray!60}6 & \cellcolor{gray!20}2 & \cellcolor{gray!40} 2 & \cellcolor{gray!10} 0 & \cellcolor{gray!50} 1 & 0 & 0&0&0 \\ \hline
			\bf 11 & 3 & 8 & 19 & \cellcolor{gray!60}15 & \cellcolor{gray!20}4 & \cellcolor{gray!40} 2 & \cellcolor{gray!10} 0 & \cellcolor{gray!50} 1 & 0&0&0 \\ \hline
			\bf 12 & 8 & 37 & 88 & 76 & \cellcolor{gray!60}34 & \cellcolor{gray!20}7 & \cellcolor{gray!40}3 & \cellcolor{gray!10} 0 & \cellcolor{gray!50} 1 &0&0\\ \hline
			\bf 13 & 21 & 147 & 390 & 435 & 218 & \cellcolor{gray!60}61 & \cellcolor{gray!20}10 & \cellcolor{gray!40} 3 & \cellcolor{gray!10}0 & \cellcolor{gray!50} 1&0\\ \hline
			\bf 14 & 67 & 550 & 1758 & 2405 & 1576 & 505 & \cellcolor{gray!60}106 &  \cellcolor{gray!20}14 & \cellcolor{gray!40}4 & \cellcolor{gray!10} 0& \cellcolor{gray!50}1\\ \hline
		\end{tabular}
		
\smallskip
{\sl 
\begin{quote}
Notes: Each shaded diagonal sequence indicates enumeration data for a fixed level (The level of a graph is $n-t$). Additional details on the code that produced the data in the table are available in~\cite{LaradjiDTDoK5, code}.
\end{quote}}

\label{tbl:ordertri}
\end{table}

To prove the generating function formulas, we connect the number of triangles to zigzags and determine some key structural identities. These are used to divide the analysis into several cases. 

The first observation we need is that all triangles appearing in $K_5$ descendants appear in zigzags (potentially including zigzags with only one triangle).
\begin{lemma}\label{lem zigzag shape}\mbox{}
  \begin{itemize}
    \item If $G$ is a $K_5$ descendant of order greater than $5$, then neither $K_4$ nor the triple triangle
      $K_{3,1,1}$ are subgraphs of $G$.
    \item If $G$ is a $K_5$ descendant of order greater than~$6$, then the triangles of $G$ are partitioned into maximal zigzag subgraphs of $G$.  The zigzag subgraphs may be 1-zigzags or $Z_n^*$s, and the $Z_n^*$s may be linked at their ends into open or closed chains with at least two zigzag pieces in each closed chain that is not a 1-zigzag.  The maximal zigzags are otherwise disjoint. 
  \end{itemize}
\end{lemma}

\begin{proof}\mbox{}
  \begin{itemize}
  \item 	Observe that since $K_5$ is completed primitive and double triangle expansion preserves completed primitivity, all $K_5$-descendants are completed primitive. Let $G$ be a $K_5$-descendant of order $>5$. If $K_4$ is a subgraph of $G$, then $G$ has a non-trivial internal $4$-edge-cut, and so it is not internally $6$-edge-connected, and so $G$ is not completed primitive, a contradiction.
    
    Suppose now that $K_{3,1,1}$ is a subgraph of $G$. Then $G$ has a $3$-vertex-cut. By Theorem \ref{thm:ancestor}, we can find the ancestor of $G$ by performing all double triangle reductions first. Since $K_5$ has no $3$-vertex-cut, it cannot be the ancestor of $G$, a contradiction.
  \item $K_5$ descendants are 4-regular, connected, simple graphs and are also completed primitive.  4-regularity implies that at most six triangles are incident with any vertex $v$ of $G$.  If there are $5$ or more then $v$ and three of its neighbours form a $K_4$ subgraph contradicting the assumptions.  If there are $4$ triangles incident with $v$ then either there is a $K_4$, or there is a wheel with four spokes centred at $v$; in the second case either $G$ is the octahedron which is excluded by order, or $G$ has a nontrivial 4-edge-cut detaching the wheel and hence is not completed primitive, all of which lead to contradictions.

    Therefore, each vertex of $v$ is incident to at most three triangles.  Since there are no triple triangles, locally the triangles in the neighbourhood of $v$ have a zigzag form.  If there is a triangle with another triangle on each edge then $G$ has a 3-vertex cut, but as in the proof of the first point this is also impossible.  Therefore there is no branching of triangle configurations, and what remains is zigzags which may join at their ends.  Finally, suppose a zigzag has one end joined to its other end but is not a 1-zigzag.  Then $G$ has a 2-edge cut which again contradicts completed primitivity, so all closed chains have at least two zigzags.  Therefore, $G$ is as described in the statement.
  \end{itemize}
\end{proof}

\subsection{Levels 0 to 3}
The first two levels are straightforward, as there is exactly one graph of level 0 at every (sufficiently large) size, and none of level 1. For higher levels we first determine template shapes, and reduce the problem to the enumeration of integer partitions under restrictions on the number of parts. 

\begin{prop}\label{prop:level0}
  There is exactly one $K_5$ descendant of level zero for each order $\geq 7$. These are precisely the 1-zigzags of order $\geq 7$.
\end{prop}

\begin{proof}
  Let $G$ be a $K_5$ descendant of level $0$. While it is true that $1$-zigzags of order $< 7$ might have additional triangles, 1-zigzags of order more than $7$ have level 0. Each double triangle reduction of a 1-zigzag of order $\geq 7$ gives the 1-zigzag of size one smaller, until reaching the octahedron which is a $K_5$ descendant itself. By Lemma~\ref{lem zigzag shape} the triangles of $G$ decompose into zigzags.  Multiple 1-zigzags are disconnected, and any other triangle configuration made of zigzags has more vertices than triangles, thus there are no other level $0$ $K_5$ descendants.
\end{proof}

There is an alternative inductive proof of this result in \cite[Proposition 4.25]{LaradjiDTDoK5} which makes use of the manner in which double triangle expansions can change the level.  Section~\ref{sec:results} relies on  similar techniques. 

\begin{lemma}\label{lem level by zigzag count}
  Suppose $G$ is a $K_5$ descendant of level $> 0$. Suppose that $G$ has $k$ maximal zigzag pieces, that there are $\ell$ vertices which are ends shared by two of the zigzag pieces, and that $m$ vertices of $G$ are in no triangles.  Then the level of $G$ is
  \[
  2k - \ell + m.
  \]
\end{lemma}

\begin{proof}
  By Lemma~\ref{lem zigzag shape} the triangles of $G$ decompose into zigzags.  However, $G$ cannot contain 1-zigzags since, by 4-regularity, each 1-zigzag is a connected component and 1-zigzags are level 0.  Therefore, all of $G$'s zigzags are non-cyclic (though, there may be cyclic chains of two or more  zigzags).

  A non-cyclic zigzag has two more vertices than triangles, so each zigzag contributes 2 to the level, except that this counts the $\ell$ identified vertices twice, and the $m$ vertices in no triangles also contribute to the level.  This gives the formula in the statement.
\end{proof}

\begin{prop}\label{prop:level1}
There are no level $1$ $K_5$ descendants.
\end{prop}

\begin{proof}
  Suppose $G$ were a $K_5$ descendant of level $1$.  Use the notation as in Lemma~\ref{lem level by zigzag count}.  $G$ must have at least two triangles since it came from a double triangle expansion, and so $k\geq 1$. (In fact, as we'll show in Section~\ref{sec:results}, $G$ must have at least four triangles.)  If $k=1$ then $\ell=0$ since there are no other zigzags to share ends with and so $L(G) = 2k+m \geq 2 > 1$ which is a contradiction.  If $k>1$ then $\ell \leq k$ as at most every end is shared, so $L(G) = 2k-\ell+m \geq k+m > 1$ which is again a contradiction.  Therefore no $K_5$ descendants of level 1 exist.
\end{proof}

Again an alternate proof of this result using induction and the ways that different double triangle expansions can affect the level is given in Proposition 4.26 of the thesis \cite{LaradjiDTDoK5} of one of us.

For the enumeration of $K_5$ descendants of levels 2 and 3 we end up enumerating certain integer partitions.  These cases are also a good preparation for the level 4 case.

\begin{prop}\label{prop:level2}
  The $K_5$ descendants of level 2 are precisely the completed primitive 2-zigzag graphs of order $\geq 8$.  The generating function for the number of these graphs, counted by their number of vertices is
  \[
  \frac{x^8}{(1-x)(1-x^2)}.
  \]
\end{prop}

Expanding the generating function, the counting sequence of level 2 $K_5$ descendants begins $(0, 0, 1, 1, 2, 2, 3, 3, \ldots )$, which agrees with the direct enumeration of small cases.

\begin{proof}
  Let $G$ be a $K_5$ descendant of level 2.  Using the notation as in Lemma~\ref{lem level by zigzag count}, we have $2k-\ell+m = 2$.  By direct counting of small graphs we see that there are no level 2 $K_5$ descendants of order $< 8$.

  If $k=1$ then $\ell=0$ since there are no other zigzags to share ends with and so we must also have $m=0$.  However the only way to add edges to a single zigzag to get a 4-regular simple graph is to build a 1-zigzag, but 1-zigzags of order $\geq 8$ have level $0$, so $k=1$ is impossible.

  Since $k>1$ then $\ell \leq k$ as at most every end is shared so $k+m\leq 2$.  Therefore $k=2$ and $m=0$ which also implies $\ell=2$.  So $G$ must consist of two zigzags joined at both of their ends; that is $G$ is a 2-zigzag.

  There is a unique way to add edges to two zigzags joined at both their ends to obtain a completed primitive 4-regular simple graph with no additional triangles.  This is illustrated on the left hand side of Figure~\ref{fig completing the 2-zigzag}.  Furthermore, all such graphs are $K_5$ descendants, as iteratively double triangle reducing each zigzag we obtain the octahedron, as illustrated on the right hand side of Figure~\ref{fig completing the 2-zigzag}, which is a $K_5$ descendant.

\begin{figure}
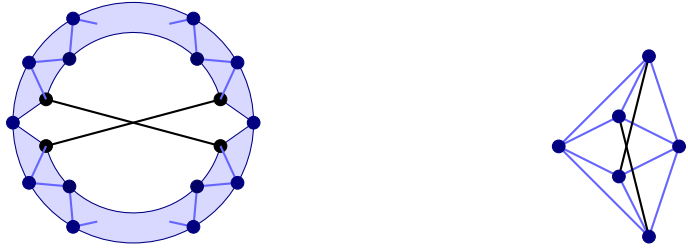

\mbox{}\hfill\leveltwo{.4}\hfill
\leveltworeduction{.4}\hfill\mbox{}
\caption{The unique 2-zigzag template and its double triangle reduction. The shaded blocks represent zigzags Note that depending on parity the top and bottom zigzags on the left may connect with a twist or without.}
\label{fig completing the 2-zigzag}
\end{figure}

  It remains to enumerate 2-zigzags.  If the two zigzags each have at least 3 triangles then the additional edges do not create any new triangles, while if either zigzag has only one or two triangles then additional triangles are created and the result is no longer a 2-zigzag.  Thus we are enumerating the possible pairs of lengths, each $\geq 3$, for the two zigzags.  Since we are enumerating up to isomorphism, this is the same as enumerating integers $z_1$ and $z_2$ subject to the constraint
  \[
  z_1 \geq z_2 \geq 3
  \]
  By standard techniques, see~\cite[Section I.3.1]{FlSe09}, this is the same as enumerating partitions into two parts, each at least 3, or equivalently enumerating partitions with at least three parts all of which are 1s or 2s.  This has generating function
  \[
  \frac{x^6}{(1-x)(1-x^2)}
  \]
  counting by number of triangles.  We are interested in generating functions which mark the number of vertices, hence we multiply by $x^2$ since each of these graphs has two more vertices than triangles. The generating function for $K_5$ descendants of level two, counting by number of vertices is
  \[
  \frac{x^8}{(1-x)(1-x^2)}.
  \]
\end{proof}

An alternate proof of this result using a vector encoding for the ways that zigzags can connect along with induction and the ways that different double triangle expansions can affect the level is given in Proposition 4.27 of the thesis \cite{LaradjiDTDoK5} of one of us.

\begin{prop}\label{prop:level3}
  The generating function for the number $K_5$ descendants of level $3$, counted by their number of vertices, is
  \[
  \frac{x^9(1+x^2)}{(1-x)^3(1+x+x^2)}.
  \]
\end{prop}

Expanding the generating function, the counting sequence of level 3 $K_5$ descendants begins $(0, 0, 1, 2, 4, 7, 10, 14, \ldots )$, which agrees with the direct enumeration of small cases in Table~\ref{tbl:ordertri}.

Note that another way to write this function is 
\[
  \frac{x^9(1+x^2)}{(1-x)^3(1+x+x^2)}=\frac{x^9(1+x^2)}{(1-x)^2(1-x^3)}= \frac{x^9(1+x^2)(1+x)}{(1-x)(1-x^2)(1-x^3)},
\]
the latter two of which transparently have the form of partition generating functions, and hence are suggestive towards our approach.

\begin{proof}
  Let $G$ be a $K_5$ descendant of level 3.  Using the notation as in Lemma~\ref{lem level by zigzag count}, we have $2k-\ell+m = 3$.

  If $k=1$ then as in the previous proofs, $\ell=0$ and so $m=1$.  However, there are no ways to add edges to a single zigzag and an additional vertex to get a 4-regular simple graph without also introducing additional triangles.  This is because the additional vertex must connect to all four of the vertices of the zigzag which are not already of degree 4, but this immediately creates two new triangles at each end.  Therefore $k=1$ is impossible.

  Since $k>1$, then as in the previous proofs $\ell \leq k$, so $k+m\leq 3$.  There are three possibilities: $(k=2, m=1, \ell=2)$, $(k=2, m=0, \ell=1)$, $(k=3, m=0, \ell=3)$.  For each of these three possibilities we exhaustively find all possible ways to add edges to the zigzags and additional vertex, if present, to get a 4-regular simple graph.  For each way to obtain a 4-regular graph, we can replace each zigzag with a zigzag on one or two triangles and this is a double triangle reduction of the graph we started with.  However, each of these reduced graphs is a single finite graph, rather than a family given by different sizes of the zigzags, and so each reduced graph can be checked to determine if it is a $K_5$ descendant by reducing any remaining double triangles.  Both of these steps are finite and automatable, and will remain so for any fixed value of the level.

  We will now proceed to follow this plan in the present situation.

  \textbf{Case 1: $(k=2, m=1, \ell=2)$}.  The two zigzags must be joined at both ends to get $\ell=2$, that is they form a closed zigzag chain.  This subgraph has four vertices of degree three, these are the four chord vertices of the zigzags, unless one of the zigzags contains only one triangle in which case there is instead a single 2-valent vertex in this zigzag that is not at either end.  To give a 4-regular graph, the additional vertex from $m=1$ must join to each of these lower degree vertices, bringing each of their degrees up to 4.  This is the only possibility for $G$ with $(k=2, m=1, \ell=2)$.  However, this graph is not a $K_5$ descendant, as if each zigzag has at least two triangles then reducing each zigzag to two triangles gives the graph in Figure~\ref{fig not k5} which has triple triangles, while if either zigzag has only one triangle then we have a double edge.  Therefore this case does not occur.

\begin{figure}
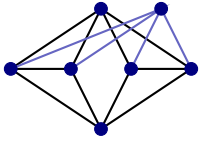

\nondescendant{.4}

\end{tikzpicture}
    \caption{The result of reducing a graph with $(k=2, m=1, \ell=2)$ and with each zigzag of length at least 2.  This is not a $K_5$ descendant.}\label{fig not k5}
  \end{figure}

  \textbf{Case 2: $(k=3, m=0, \ell=3)$}.  Here the three zigzag pieces must form a closed zigzag chain.  Consider first the situation where each zigzag has at least three triangles in it.  The zigzag chain has six 3-valent vertices (the six chord vertices of the zigzags) and all other vertices are already 4-valent.  We need to add three edges, matching the six chord vertices by pairs, to obtain a 4-regular graph.  If we match the two chord vertices in the same zigzag then we get a graph with a 4-edge cut and so a graph which is not completed primitive, hence not a $K_5$ descendant.  If we match two chord vertices which are both adjacent to the common end vertex shared by two zigzags then we create a new triangle.  Up to rotational symmetry this leaves only two possibilities:
  
  \centerline{\levelthreea{.3} \quad \levelthreeb{.3}. }
  
  In these diagrams the shaded sections represent zigzags.  The indicated vertices are the end and chord vertices of each zigzag, and the chord vertices are joined by additional edges as indicated.  By reducing each zigzag to a triangle we obtain the octahedron for both possibilities:
\begin{center}
 \begin{tikzpicture}[rotate=60,scale=.6]
\foreach \x in {30,150,270} {
  \fill[white,draw=black] (\x-60:1.5) arc (\x-60:\x+60:1.5) -- (\x:.5) -- cycle;
}
  \draw[thick](30:.5) -- (150:.5);
  \draw[thick](150:.5) -- (270:.5);
  \draw[thick](270:.5) -- (30:.5);
  \filldraw(30:0.5) circle (3pt);
  \filldraw(150:0.5) circle (3pt);
  \filldraw(270:0.5) circle (3pt);
   \filldraw(90:1.5) circle (3pt);
  \filldraw(210:1.5) circle (3pt);
  \filldraw(330:1.5) circle (3pt);
\end{tikzpicture}
\end{center}
This is a $K_5$ descendant.

  If any zigzag contains only one or two triangles then the situation is slightly different.  Call such zigzags \emph{short} zigzags.  If a zigzag contains only one triangle then, rather than two 3-valent vertices, there is one 2-valent vertex distinct from the ends, while if a zigzag contains exactly two triangles, then there are still two 3-valent vertices, however there is an automorphism of the zigzag which swaps them, so for the purposes of enumerating ways to add edges to obtain a 4-regular graph these two vertices are indistinguishable.  Short zigzags will be indicated in the diagrams with shorter shaded segments with only one marked inner vertex.  This one vertex needs to be connected to two other vertices of the diagram.  This one vertex either represents the one 2-valent vertex not at the ends in the zigzag with one triangle or represents together the two indistinguishable 3-valent vertices of the two triangle zigzag.  As in the case with all zigzags of length at least 3, a vertex cannot be connected within the same zigzag nor to the nearest vertex in the nearest zigzag.  As a consequence two short zigzags are not possible in the same graph in the $(k=3, m=0, \ell=3)$ case, and there is only one possibility with one short zigzag:
  \centerline{\levelthreec{.3}}. This one possibility does give $K_5$ descendants as reducing all zigzags to single triangles gives the same graph (the octahedron) as obtained with no short zigzags.

  For each possibility we now want to count how many such graphs there are.  For the two possibilities with no short zigzags, let the three lengths be $z_1\geq z_2\geq z_3\geq 3$, where length is the number of triangles.  Suppose first the $z_i$ are distinct.  Then in the less symmetric possibility (on the left in Table~\ref{tab L=3 not small}) there are three ways to assign the $z_i$s to the lengths of the zigzags based on whether the largest, middle, or smallest, is assigned to the distinguishable zigzag.  In the more symmetric possibility (on the right in Table~\ref{tab L=3 not small}) there is only one way to assign the $z_i$.  Arguing similarly when some $z_i$s are equal we obtain Table~\ref{tab L=3 not small} which collects how many ways there are to assign the $z_i$s, subject to the indicated constraints, for each possibility.  The situation with a short zigzag is similar, but we have $z_1\geq z_2 \geq 3$ as the lengths of the two non-short zigzags, and $z_3\in\{1,2\}$ for the short zigzag.  Table~\ref{tab L=3 small} collects how many ways there are to assign the lengths in each case.  The table indicates $2\geq z_3\geq 0$; so far we have $2\geq z_3\geq 1$ and the case $z_3=0$ is discussed below.
  
  \textbf{Case 3: $(k=2, m=0, \ell=1)$}.
  Here we have two zigzags joined at one end with their other ends free, that is an open zigzag chain of two zigzags.  This leaves two 2-valent vertices at one end of each zigzag, which we must get up to degree 4 by adding edges.  These cannot be joined within their zigzag for the same reason as above.  To avoid creating new triangles, a 2-valent vertex cannot be joined to two adjacent vertices on the other zigzag.  This implies that neither zigzag can be short.  Also, a 2-valent vertex cannot join to the two non-end 3-valent vertices of another zigzag, as if so then double triangle reducing the other zigzag down to a single triangle would give a double edge.  This leaves 
  \[
  \begin{tikzpicture}[scale=.3]
\coordinate  (A0) at (70:\rout);
\coordinate  (B0) at (110:\rout);

\arcarrow{110}{160}{blue!30,draw = blue!50!black}{1} 
\arcarrow{270}{160}{blue!30,draw = blue!50!black}{2} 
 
\foreach \x in  {270} {
  \filldraw[blue!50!black](\x:\rout) circle (6pt);
}
  \draw[thick](A0) -- (B0);
  \filldraw(A0) circle (6pt);
  \filldraw(B0) circle (6pt);
\end{tikzpicture}\]
  
  as the only possibility.
This case has an edge between the two free ends of the zigzags, leaving one additional non-zigzag edge required for each of these vertices.  Further, it suffices to know where the other two ends of these edges are, without keeping track of which came from which zigzag, because the fact that edges cannot stay within a zigzag disambiguates.  Thus this case carries the same information as the case in Table~\ref{tab L=3 small}, but now what is behaving as the short zigzag has 0 triangles.  This gives the case $z_3=0$ in Table~\ref{tab L=3 small}, and joining the ends of zigzags with an edge is what we mean by a zigzag of length 0 in our diagrams.

  Now we have what we need to find the generating function.  Considering Table~\ref{tab L=3 not small} and Table~\ref{tab L=3 small}, the constraints on the $z_i$ determine generating functions for partitions by standard techniques, as in the proof of the previous proposition.  These are given in the last column of the tables.  Each generating function is weighted by its row sum as the row sum indicates how many non-isomorphic ways these lengths can be assigned to a zigzag configuration giving a $K_5$ descendant.  We obtain

  \begin{table}
  \begin{tabular}{lcccl}
 Partition& \levelthreea{.1} &\levelthreeb{.1} & Row sum & OGF \\
  $z_1>z_2>z_3\geq 3$ & 3& 1 &4&$x^{12}((1-x^3)(1-x^2)(1-x))^{-1}$\\
  $z_1=z_2>z_3\geq 3$ & 2& 1 &3&$x^{11}((1-x^3)(1-x^2))^{-1}$ \\
  $z_1>z_2=z_3\geq 3$ & 2& 1 &3&$x^{10}((1-x^3)(1-x))^{-1}$\\
  $z_1=z_2=z_3\geq 3$ & 1& 1 &2&$x^9(1-x^3)^{-1}$\\
  \end{tabular}
  \caption{Conditions on the numbers of triangles in each zigzag, multiplicities, and generating functions for the level 3 configurations with no small zigzags.}\label{tab L=3 not small}
  \end{table}
\begin{table}
  \begin{tabular}{lcl}
 Partition & \levelthreec{.1} & OGF \\
  $z_1>z_2>2 \geq z_3\geq 0$ &  1 &$x^7(1+x+x^2)((1-x^2)(1-x))^{-1}$\\
  $z_1=z_2>2 \geq z_3\geq 0$ & 1 &$x^6(1+x+x^2)((1-x^2))^{-1}$ \\
  \end{tabular}
  \caption{Conditions on the numbers of triangles in each zigzag, multiplicities, and generating functions for the level 3 configurations with one small zigzag.}\label{tab L=3 small}
  \end{table}

\begin{align*}
  & 4 \frac{x^6}{(1-x^3)(1-x^2)(1-x)} + 3 \frac{x^5}{(1-x^3)(1-x^2)} + 3 \frac{x^4}{(1-x^3)(1-x)} + 2 \frac{x^3}{1-x^3} \\
  & \qquad + \frac{x^6}{(1-x^3)(1-x^2)(1-x)} + \frac{x^5}{(1-x^3)(1-x^2)}\\
  & = \frac{x^6(1+x^2)}{(1-x)^3(1+x+x^2)}
\end{align*}
for the generating function of level 3 $K_5$ descendants counted by number of triangles, or equivalently
\[
  \frac{x^9(1+x^2)}{(1-x)^3(1+x+x^2)}
\]
 for the generating function of level 3 $K_5$ descendants counted by number of vertices. 
\end{proof} 

Again, an alternative proof is found in~\cite[Proposition 4.30]{LaradjiDTDoK5}. The configuration diagrams give an alternate way of capturing the information in the vector encodings in the proof of \cite{LaradjiDTDoK5}.  Vector encodings will be discussed in Section~\ref{sec:results}.

\subsection{Level 4}
Using the initial terms of the counting sequence, Laradji guessed a form for the generating function of $K_5$ descendants of level 4~\cite[Conjecture 4.31]{LaradjiDTDoK5}. We now prove this conjecture by a combinatorial argument, proceeding in a manner similar to the level 2 and 3 cases. 

\begin{prop}\label{prop:level4}
The generating function for the number of level $4$ $K_5$ descendants, counted by their number of vertices, is
  \[
  \frac{x^{20} - x^{19} +3x^{18} - 3x^{17} + 4x^{16} + 4x^{14} + 3x^{13} + 6x^{12} + 3x^{11} + 4x^{10} + x^9}{(1-x)^4(1+x)^2(1+x^2)}
  \]
\end{prop}

Expanding the generating function, the counting sequence of level 4 $K_5$ descendants begins $(0, 1, 6, 15, 34, 61, 106, 162, 246, \ldots)$, which agrees with the direct enumeration of small cases.

\begin{proof}
  Let $G$ be a $K_5$ descendant of level 4.  Using the notation as in Lemma~\ref{lem level by zigzag count} we have $2k-\ell+m = 4$.  As in the level 3 proof let a \emph{short} zigzag be one of size $<3$.

    If $k=1$ then as in the previous proofs, $\ell=0$ and so $m=2$.  However, there are no ways to add edges to a single zigzag and two additional vertices to get a 4 regular simple graph without also introducing additional triangles, so $k=1$ is impossible.

    Since $k>1$, then as in the previous proofs $\ell \leq k$, so $k+m\leq 4$.  There are six possibilities,
    \begin{enumerate}
    \item $(k=4, m=0, \ell=4)$,
    \item $(k=3, m=0, \ell=2)$,
    \item $(k=3, m=1, \ell=3)$,
    \item $(k=2, m=0, \ell=0)$,
    \item $(k=2, m=1, \ell=1)$, and
    \item $(k=2, m=2, \ell=2)$.
    \end{enumerate}
    As in the level 3 case, for each of these three possibilities we exhaustively find all possible ways to add edges to get a 4-regular simple graph and for each of these ways replace each zigzag with a zigzag on one or two triangles giving finite graphs for which we can check directly whether or not they are $K_5$ descendants by reducing any remaining double triangles.

Although the case analysis is tedious, it is a finite, and 
 automatable process. We are able to reduce the number of cases as we did in the level 3 case,  with the help of the following two observations:  No added edge can have both ends in the same zigzag (to avoid contradicting completed primitivity); and no added edge can connect the vertices adjacent to the shared end in two zigzags sharing an end (to avoid creating new triangles).

This removes two cases immediately. The last possibility, $(k=2, m=2, \ell=2)$ is impossible as there are not enough vertices to connect to the two vertices that are in no triangles.  For $(k=3, m=1, \ell=3)$ two edges out of the vertex in no triangles must go to the same zigzag and so this is not a $K_5$ descendant as double triangle reducing this zigzag would give a double edge.  Four possibilities remain.
    
For the first remaining possibility, $(k=4, m=0, \ell=4)$, either there are two closed chains of 2 zigzags each or a closed chain of four zigzags.  The first possibility is not completed primitive as four edges go one from each closed chain, and so we must have a closed chain of four zigzags.  (In fact we'll see in Corollary~\ref{lem K5 desc is pseudo desc} that a $K_5$ descendant can have a closed chain only if it has exactly one and no vertices in no triangles and hence is an $n$-zigzag.)
    
    If there are no short zigzags we obtain the configurations of zigzags in Table~\ref{tab L=4 four not small}.  As in the level 3 case, we need multiplicities to account for how many ways the lengths of the zigzags $z_1\geq z_2\geq z_3\geq z_4$ can be assigned in each configuration, where length is the number of triangles.  Again we do this by separating based on which of the $z_i$ are equal.  In this case we further break it up by the cyclic order of the distinct lengths.  This is indicated in the column of the table labelled ``Pattern''.  The results of these computations are in Table~\ref{tab L=4 four not small}.

If any of the zigzags are short, some of the templates coincide.  
Remaining in the first possibility for the moment, if there is one short zigzag then we obtain the configurations of zigzags in Table~\ref{tab L=4 one small}.  Note that in each of these configurations, at least one of the additional edges out of the short zigzag goes to an adjacent zigzag.  This implies that, as in the proof of case 3 at level 3, if we have an edge joining the ends two zigzags then this carries the same information as a short zigzag with the fact that edges cannot stay within a zigzag disambiguating which additional edge goes to which end of the edge joining the ends of the zigzags.  In this way we can again think of an edge joining the ends of two zigzags as a size 0 zigzag joining the two zigzags.  Therefore Table~\ref{tab L=4 one small} takes care of both $(k=4, m=0, \ell=4)$ with one short zigzag and $(k=3, m=0, \ell=2)$ when there is an edge joining the two free zigzag ends and none of the three zigzags are short.

    The case of $(k=4, m=0, \ell=4)$ with two short zigzags gives the configurations in Table~\ref{tab L=4 two small} and again at least one additional edge out of each short zigzag goes to an adjacent zigzag so we can also capture size 0 zigzags in this case.  This additionally takes care of $(k=3, m=0, \ell=2)$ when there is an edge joining the two free zigzag ends and one of the three zigzags is short as well as $(k=2, m=0, \ell=0)$ when there are two edges pairing the ends of each zigzag to the other, and $(k=2, m=1, \ell=1)$ when a path with two edges goes from one free zigzag end, through the non-triangle vertex, to the other free zigzag end.

Furthermore, more than two short zigzags in $(k=4, m=0, \ell=4)$ or more than one short zigzag in $(k=3, m=0, \ell=2)$ when there is an edge joining the two free zigzag ends, is impossible.  This is because the only ways to add the additional edges so as not to make additional triangles results in two edges from the inside of one zigzag to the inside of another and this cannot be a $K_5$ descendant as double edges result upon reducing the zigzags.

    Additionally, the only ways to add edges in $(k=2, m=0, \ell=0)$ (without creating additional triangles or connecting within a zigzag) pair the ends of the zigzags with an edge, so this case is now fully dealt with.  Similarly, the only way to add edges in $(k=2, m=1, \ell=1)$ (without creating additional triangles or connecting within a zigzag) leads to a path with two edges from one free zigzag end, through the non-triangle edge, to the other free zigzag end, so this case is fully dealt with as well.

    It remains to consider $(k=3, m=0, \ell=2)$ when there is no edge joining the two free zigzag ends.  With no short zigzags this is given in Table~\ref{tab L=4 three big}.  With short zigzags there is one extra consideration. A zigzag with one triangle is just a triangle.  If it joins to only one other zigzag then the remaining two corners are indistinguishable, so the free end and the inner 2-valent vertex are the same.  With this in mind, all the cases with one of the end zigzags having length 1 have already been counted among the cases with size 0 parts, as one of the two 2-valent vertices must connect to the other free zigzag end.  Therefore, when short zigzags appear with a free end in Tables \ref{tab L=4 two big one small part 2}, \ref{tab L=4 one big two small}, and \ref{tab L=4 all small} they must have size 2.  Additionally, size 0 zigzags cannot appear in any of these cases as there are simply no valid ways to add edges for such configurations when the size 0 zigzag is at one end, while when the size 0 zigzag is between the larger two zigzags then swapping the direction of one of the larger zigzags we can reinterpret this as being one of the case $(k=2, m=0, \ell=0)$ configurations with the ends of each zigzag paired to the other.

Finally, using the constraints on the lengths of the zigzags we calculate the generating functions for each row by standard techniques, as in the level 3 case, and sum these weighted by the row sums of the weights in the tables.  Doing so gives
\[
  \frac{x^{16} - x^{15} +3x^{14} - 3x^{13} + 4x^{12} + 4x^{10} + 3x^{9} + 6x^{8} + 3x^{7} + 4x^{6} + x^5}{(1-x)^4(1+x)^2(1+x^2)}
\]
counted by number of triangles.  Multiplying by $x^4$ gives the generating function counted by number of vertices, which is the desired result.

\subsection{Asymptotic growth of the number of graphs}
The asymptotic number of graphs in each class is straightforward to compute since the generating functions are rational. If $g_L(n)$ is the number of $K_5$ descendants on $n$ vertices and level $L$, then $g_2(n)\sim n/2$
$g_3(n)\sim n^2/3$ and $g_4(n)\sim \frac{25}{48}n^3$ following classic techniques of coefficient analysis (see~\cite[Section IV.5]{FlSe09}). We conjecture that $g_k(n)\sim C_k n^{k-1}$ for rational constants $C_k$. This is not obvious from the initial definition, but the subexponential growth follows from the main template with $k$ zigzags. It is more challenging to determine if the limit of the $C_k$ converges as the $k$ tend to infinity, and if so, to what.  


\begin{table}
\caption{Four large zigzags}
\begin{tabular}{llccccccccc}
Partition & Pattern 
&\fourbigzag{A1}{B3}{B1}{A0}{A2}{A3}{B2}{B0}{.15}
&\fourbigzag{A0}{A2}{B0}{B1}{A1}{A3}{B2}{B3}{.15}
&\fourbigzag{A0}{A3}{B3}{A1}{B0}{A2}{B1}{B2}{.15}
&\fourbigzag{A0}{B1}{B0}{B2}{A1}{A3}{A2}{B3}{.15}
&\fourbigzag{A1}{A2}{B1}{B0}{B2}{B3}{A3}{A0}{.15}
&\fourbigzag{A0}{B1}{B0}{A3}{A1}{A2}{B2}{B3}{.15}
&\fourbigzag{A0}{B1}{B0}{A3}{A1}{B2}{A2}{B3}{.15}\\[2mm]
$z_1>z_2>z_3>z_4$ &1234&8 &4&4&4&2&2&1\\
&1243&8 &4&4&4&2&2&1\\
&1324&8 &4&4&4&2&2&1\\
\hline
$z_1=z_2>z_3>z_4$ & 1134 & 8 & 4&4&4&2&2&1\\
& 1314 & 4&2&2&3&2&1&1\\
$z_1>z_2=z_3>z_4$ & 1224 & 8&4&4&4 &2&2&1\\
&1242 & 4&2&2&3&2&1&1\\
$z_1>z_2>z_3=z_4$& 1233  & 8&4&4&4 &2&2&1\\
&1242 & 4&2&2&3&2&1&1\\\hline
$z_1=z_2>z_3=z_4$ & 1133 &4&2&3&2&1&2&1\\
& 1313 & 2&2&1&2&2&1&1\\\hline
$z_1=z_2=z_3>z_4$ & 1114 & 4& 2&2&3&2&1&1\\
$z_1>z_2=z_3=z_4$ & 1444 &  4& 2&2&3&2&1&1\\ \hline
$z_1=z_2=z_3=z_4$ & 1111 & 1&1&1&1&1&1&1\\ \hline
\end{tabular}
\smallskip

\label{tab L=4 four not small}
\end{table}

\end{proof}

\begin{table}
\caption{Three large zigzags, one short (potentially degenerate) zigzag}
\label{tab L=4 one small}
\begin{tabular}{lcccccc}
Partition&Total
&\threebigzagosm{A0}{B1}{A0}{B2}{B3}{A1}{A2}{A3}{.15}
&\threebigzagosm{A0}{B1}{A0}{A2}{B3}{B2}{A1}{A3}{.15}
&\threebigzagosm{A0}{B1}{A0}{B2}{A2}{B3}{A1}{A3}{.15}
&\threebigzagosm{A0}{A3}{A0}{B1}{A1}{A2}{B2}{B3}{.15}
&\threebigzagosm{A0}{A3}{A0}{B1}{A1}{B2}{A2}{B3}{.15}
\\[2mm]
$z_1>z_2>z_3>2\geq z_4\geq 0$ &24  &6&6&6&3&3\\
$z_1=z_2>z_3>2\geq z_4\geq 0$ &13  &3&3&3&2&2\\
$z_1>z_2=z_3>2\geq z_4\geq 0$ &13  &3&3&3&2&2\\
$z_1=z_2=z_3>2\geq z_4\geq 0$ &5   &1&1&1&1&1\\
\end{tabular}
\smallskip

\end{table}

\begin{table}\caption{Two large zigzags, two short (potentially degenerate) zigzags.}\label{tab L=4 two small}
\begin{tabular}{lllccc}
Partition&Pattern &Total
&\twolongtwoshortc{.15}
&\twolongtwoshortb{.15}
&\twolongtwoshorta{.15}
\\[2mm]
$z_1 >z_2  >2\geq z_3 > z_4 \geq 0$ 
&1234 &1  &1&0&0\\ 
&1243 &1  &1&0&0\\
&1324 &3  &0&1&2\\
$z_1 =z_2  >2\geq z_3 > z_4 \geq 0$    
&1134 &1  &1&0&0\\ 
&1314 &2  &0&1&1\\
$z_1 > z_2 >2\geq z_3 = z_4 \geq 0$ 
&1233 &1  &1&0&0\\ 
&1323 &2  &0&1&1\\
$z_1 = z_2 >2\geq z_3 = z_4 \geq 0$
&1133 &1  &1&0&0\\ 
&1313 &2  &0&1&1\\
\end{tabular}
\smallskip

\end{table}

\begin{table}
\caption{Three large zigzags, no short zigzags}\label{tab L=4 three big}
\begin{tabular}{lccccccc}
Partition&Total
&\threebigzag{B1}{B2}{B0}{B3}{B0}{A2}{A1}{A3}{.15}
&\threebigzag{B1}{A2}{B0}{B3}{B0}{B2}{A1}{A3}{.15}
&\threebigzag{B1}{B2}{B0}{A3}{B0}{A2}{A1}{B3}{.15}
&\threebigzag{A2}{B1}{B0}{B2}{B0}{A3}{A1}{B3}{.15}
&\threebigzag{A1}{B2}{B0}{B3}{B0}{A2}{B1}{A3}{.15}
&\threebigzag{A1}{A2}{B0}{B3}{B0}{B2}{B1}{A3}{.15}
\\[2mm]
$z_1>z_2>z_3>2$ &24  &6&6&3&3&3&3\\
$z_1=z_2>z_3>2$ &14  &3&3&2&2&2&2\\
$z_1>z_2=z_3>2$ &14  &3&3&2&2&2&2\\
$z_1=z_2=z_3>2$ &6   &1&1&1&1&1&1\\
\end{tabular}
\smallskip

\end{table}

\begin{table}
\caption{Two large zigzags, one short zigzag (part 1)}\label{tab L=4 two big one small part 1}
\begin{tabular}{lcccccccc}
Partition&Total
&\twobigzagosm{B1}{A2}{B0}{B3}{B0}{A2}{A1}{A3}{.15}
&\twobigzagosm{B1}{A2}{B0}{A2}{B0}{A3}{A1}{B3}{.15}
&\twobigzagosm{A1}{A2}{B0}{B3}{B0}{A2}{B1}{A3}{.15}
\\[2mm]
$z_1>z_2>2\geq z_3\geq 1$ & 4 &2 &1 &1 & & \\
$z_1=z_2>2\geq z_3\geq 1$ & 3 &1 &1 &1 & & \\
\end{tabular}
\smallskip

\end{table}

\begin{table}
\caption{Two large zigzags, one short zigzag (part 2)}\label{tab L=4 two big one small part 2}
\begin{tabular}{lcccccccc}
Partition&Total
&\twobigzagosmb{A1}{B2}{B0}{A2}{B0}{B3}{A1}{A3}{.15}
&\twobigzagosmb{A1}{A2}{B0}{B2}{B0}{A3}{A1}{B3}{.15}
&\twobigzagosmb{A1}{A2}{B0}{B2}{B0}{B3}{A1}{A3}{.15}
&\twobigzagosmb{A1}{B2}{B0}{A2}{B0}{A3}{A1}{B3}{.15}
\\[2mm]
$z_1>z_2>2\geq z_3=2$ & 8 &2 &2 &2 &2 \\
$z_1=z_2>2\geq z_3=2$ & 4 &1 &1 &1 &1  \\
\end{tabular}
\smallskip

\end{table}
\begin{table}
\caption{One large zigzag, two short zigzags}\label{tab L=4 one big two small}
\begin{tabular}{llcccc}
Partition&Total
&\onebigzigzaga{.15}
&\onebigzigzagb{.15}
&\onebigzigzagc{.15}
&\onebigzigzagd{.15}
\\[2mm]
$z_1>2, z_2=2, z_3=1$ &2  &1&1&0&0\\
$z_1>2, z_2=z_3=2$    &4  &1&1&1&1\\
\end{tabular}
\smallskip

\end{table}

\begin{table}

\caption{Zero large zigzags}\label{tab L=4 all small}
\begin{tabular}{llc}
Partition&Total
&\nobigzigzag{.15}\\[2mm]
$z_1=z_2=2\geq z_3\geq 1$ &1  &1\\
\end{tabular}
\smallskip
\end{table}

\section{Minimum number of triangles in $K_5$ descendants}
\label{sec:results}
In this section we show that the every $K_5$ descendant has at least $4$ triangles.  Let $\mbox{tri}(G)$ denote the number of triangles in the graph $G$ (e.g. $\mbox{tri}(K_5)=10$).

Double triangle expansion always produces a double triangle, and thus, trivially, any $K_5$ descendant has at least two triangles. Interestingly, it can be shown that $K_5$ descendants have at least $4$ triangles (Theorem~\ref{thm:tri4}).  This was first observed in data computed by one of us \cite{LaradjiDTDoK5, code}, and this data furthermore shows that the four triangles are not necessarily two double triangles, but may appear as one double triangle and two isolated triangles. To prove that the minimum number of triangles is 4, we will look at the effect of a double triangle reduction on the possible lengths of zigzags.

\subsection{Chain vectors}
In Section~\ref{sec:enumerative} it was important to look at how $K_5$ descendants decomposed in terms of zigzags and \emph{non-triangle vertices}, which are vertices not in any triangles.  We know by Lemma~\ref{lem zigzag shape} that all $K_5$ descendants do decompose in this way. This opens up the possibility of vector encodings for $K_5$ descendants. Although one can define a set of vectors from which a $K_5$ descendant can be completely reconstructed~\cite{LaradjiDTDoK5}, for our purposes it suffices to only consider the lengths of zigzags and breaks between chains. This information is encoded in the chain vector, which can be defined as follows. 

The \emph{length} of a zigzag is the number of its triangles. The \emph{chain vector} of a $(z_1, \cdots, z_n)$-chain $C$, where the $z_i$s are the lengths of the maximal zigzags of $C$, is $\CV(C):=(z_1, \cdots, z_n)$ if the chain is closed, and $\CV(C):=(z_1, \cdots, z_n, 0)$ if it is open. If $G$ is an $n$-zigzag, its \emph{chain vector} is that of its (only) closed chain. If $G$ consists only of open chains $C_1,\cdots,C_r$ and non-triangle vertices, its \emph{chain vector} is the concatenation of the chain vectors of its open chains $\CV(G):=(\CV(C_1)|\cdots|\CV(C_r))$ where $|$ indicates concatentation. 
By Corollary~\ref{lem K5 desc is pseudo desc}, every $K_5$ descendant of order greater than $6$ has a chain vector. Conversely, all $4$-regular graphs that are either $n$-zigzags or have all their triangles in open chains have a chain vector, whether they are $K_5$ descendants or not; such graphs are called pseudo-descendants.

The chain vector is not unique, as it depends on the choice of the direction of each chain and the ordering of the chains. In addition, $0$s in the chain vector represent a break between two chains, and so we will interpret consecutive $0$s (which come up in later arguments) as equivalent to one $0$.  That is, the following chain vectors are equivalent: $(3,3,2) = (3,2,3) = (2,3,3)$ as are $(3,0,1,2,0,0,2,0)=(3,0,1,2,0,2,0)=(0,2,0,3,0,2,1,0)$. 

Representing non-triangle vertices with negative numbers, the chain vector can be extended to a full vector encoding for $K_5$ descendants (and other pseudo-descendants). Fixing a chain vector fixes an order on the vertices which need additional edges added to obtain a 4-regular graph. These additional edges can then be indicated by how far along this order their other end is. This gives the vector encoding developed by one of us in \cite{LaradjiDTDoK5}, and software was written for working with these vector encodings \cite{code}.  Note that, while this full vector encoding determines the graph, characterizing which vectors give $K_5$ descendants remains tricky.  

The chain vector that is used here is either a list of positive integers indicating a closed $n$-zigzag with those zigzag lengths, or a list with a terminal zero and potentially additional zeros inside where the sublists delimited by the zeros are the lengths of the zigzags in the open chains.

\subsection{Chain vectors of $K_5$ descendants}

We want to characterize the effect of a single DTR on the chain vector of a $K_5$ descendant of order greater than $7$. To do this, we begin by looking at the different ways a DTE can affect the chain vector, and then use the fact that DTR is the inverse of DTE. Although we can consider the effect of DTR directly, this will increase the number of possible DTR cases, as there are certain DTR operations that cannot be performed on any $K_5$ descendant of order greater than $7$.

To determine how a single DTE affects the chain vector of a graph, it is useful to categorize DTEs based on what other triangles or edges exist or can be formed in the neighbourhood of the triangle being expanded. This is done in the following definition.

\newcommand{\setvertex}[2]{%
	\csdef{vertex#1}{#2}}
\setvertex{1}{6}
\setvertex{2}{1}
\setvertex{3}{3}
\setvertex{4}{2}
\setvertex{5}{4}
\setvertex{6}{5}
\setvertex{a}{d}
\newcommand{\vertex}[1]{v_{\csuse{vertex#1}}}

\begin{definition}[DTE type]
	Suppose a $4$-regular simple graph $G$ contains a triangle-and-edge subgraph $H$, consisting of a triangle $(\vertex{2},\vertex{4},\vertex{3})$ and an edge $(\vertex{4},\vertex{5})$, with $\vertex{2},\vertex{4},\vertex{3},\vertex{5}$ distinct. Suppose further that the triangles of $G$ are partitioned into maximal zigzag subgraphs.  Then by swapping $\vertex{2}$ and $\vertex{3}$ if necessary we may assume that $\vertex{2}$ has no common neighbour with $\vertex{4}$ other than $\vertex{3}$.
        The \emph{type} of a DTE is a $3$-tuple $(a,b,c)$, with $a,b,c$ defined as follows:
	\begin{itemize}
		\item If $\exists \vertex{6} \notin V(H)$ such that $(\vertex{4},\vertex{5},\vertex{6})$ is a triangle in $G$, then $a=1$, otherwise $a=0$.
		\item If $\exists \vertex{1} \notin V(H)$ such that $(\vertex{1},\vertex{2},\vertex{3})$ is a triangle in $G$, then $b=1$, otherwise $b=0$.
		\item If $(\vertex{3},\vertex{5})$ is an edge in $G$, then $c=1$, otherwise $c=0$.
	\end{itemize}
	\label{defn:triangletypes}
\end{definition}
Note that the triangles of $G$ being partitioned into maximal zigzag subgraphs further implies that no additional triangles appear involving only the vertices $\vertex{2}, \vertex{4}, \vertex{3}, \vertex{5}$ and $\vertex{1}$ and $\vertex{6}$ when they exist.


Let $G$ be a $K_5$ descendant with a chain vector. Any triangle in $G$ has one of the forms on the left-hand-side of Figure~\ref{fig:triangletypes}. The effect of a single DTE on the chain vector is summarized in Table~\ref{tbl:dtrcveffect} and illustrated in Figure~\ref{fig:triangletypes}. This is proved in the following lemma.


\newlength{\edgewidth}
\setlength{\edgewidth}{1.5pt}
\newlength{\doubleedgewidth}
\setlength{\doubleedgewidth}{1pt}
\newcommand{\scscaling}{0.6}
\newcommand{\dtecolor}{red}

\tikzstyle{point} = [draw, circle, very thick, fill=black, minimum size=4pt, inner sep=0pt]
\tikzstyle{labeled} = [draw,circle,very thick]
\tikzstyle{edge} = [line width=\edgewidth]
\tikzstyle{dteedge} = [line width=\edgewidth, color=\dtecolor]
\tikzstyle{nonedge} = [line width=\edgewidth, dashed]
\tikzstyle{possibleedge} = [decoration={dashsoliddouble}, decorate, line width=\doubleedgewidth]


\newlength{\decorationedgeoffset}
\setlength{\decorationedgeoffset}{0.3pt}

\pgfdeclaredecoration{dashsoliddouble}{initial}{
	\state{initial}[width=\pgfdecoratedinputsegmentlength]{
		\pgfmathsetlengthmacro\lw{1.3*\pgflinewidth}
		\begin{pgfscope}
			\pgfsetlinewidth{\pgflinewidth}
			\pgfpathmoveto{\pgfpoint{-\decorationedgeoffset}{\lw}}%
			\pgfpathlineto{\pgfpoint{\pgfdecoratedinputsegmentlength+\decorationedgeoffset}{\lw}}%
			\pgfmathtruncatemacro\dashnum{%
				round((\pgfdecoratedinputsegmentlength-3pt)/6pt)
			}
			\pgfmathsetmacro\dashscale{%
				(\pgfdecoratedinputsegmentlength+2*\decorationedgeoffset)/(\dashnum*6pt + 3pt)
			}
			\pgfmathsetlengthmacro\dashunit{3pt*\dashscale}
			\pgfsetdash{{\dashunit}{\dashunit}}{0pt}
			\pgfusepath{stroke}
			\pgfsetdash{}{0pt}
			\pgfpathmoveto{\pgfpoint{-\decorationedgeoffset}{-\lw}}%
			\pgfpathlineto{\pgfpoint{\pgfdecoratedinputsegmentlength+\decorationedgeoffset}{-\lw}}%
			\pgfusepath{stroke}
		\end{pgfscope}
	}
}

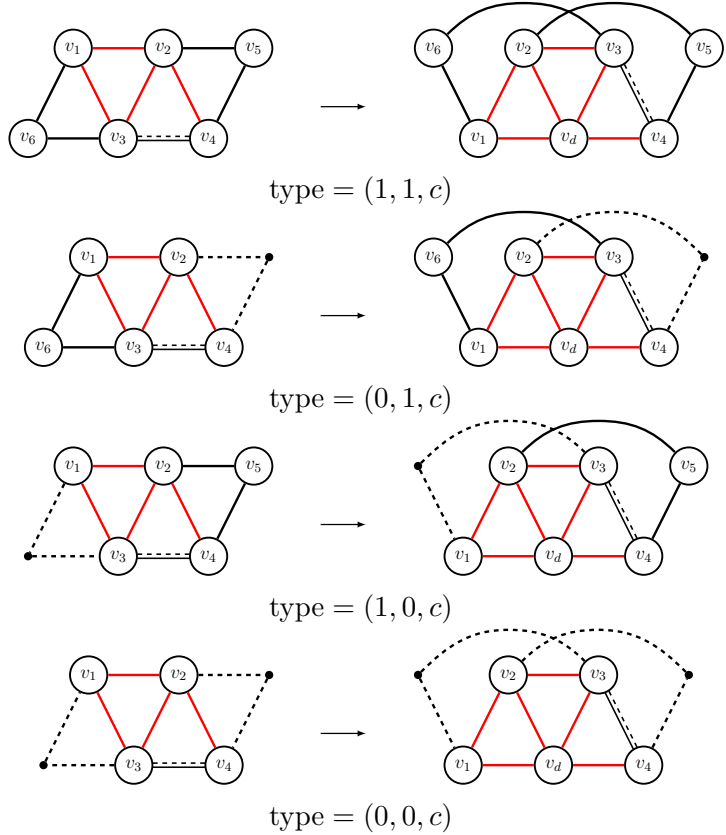
\begin{figure}
  \centering
  \subcaptionbox*{$\mbox{type}=(1,1,c)$}{
  \scalebox{\scscaling}{
  \centering
	\begin{tikzpicture}
	
	\node[labeled] (v1) at (0,0) {$\vertex{1}$};
	\node[labeled] (v2) at (1,2) {$\vertex{2}$};
	\node[labeled] (v3) at (2,0) {$\vertex{3}$};
	\node[labeled] (v4) at (3,2) {$\vertex{4}$};
	\node[labeled] (v5) at (4,0) {$\vertex{5}$};
	\node[labeled] (v6) at (5,2) {$\vertex{6}$};
	
	\draw[dteedge]
	(v2)--(v3)--(v4)--(v2)
	(v4)--(v5)
	;
	\draw[edge]
	(v1)--(v2)
	(v4)--(v6)--(v5)
	(v1)--(v3)
	;
	
	\draw[possibleedge]
	(v3)--(v5);
	
	\end{tikzpicture}
	\separator
	\begin{tikzpicture}
	\node[labeled] (v1) at (0,2) {$\vertex{1}$};
	\node[labeled] (v2) at (1,0) {$\vertex{2}$};
	\node[labeled] (v4) at (2,2) {$\vertex{4}$};
	\node[labeled] (va) at (3,0) {$\vertex{a}$};
	\node[labeled] (v3) at (4,2) {$\vertex{3}$};
	\node[labeled] (v5) at (5,0) {$\vertex{5}$};
	\node[labeled] (v6) at (6,2) {$\vertex{6}$};
	
	\draw[dteedge]
	(v2)--(v4)--(va)--(v3)--(v4)
	(v2)--(va)
	(va)--(v5)
	;
	\draw[edge]
	(v1)--(v2)
	(v5)--(v6)
	[-] (v1)  to [out=45,in=180,in looseness=1] (2cm,3cm) to [out=0,in=135,out looseness=1](v3)
	[-] (v6)  to [out=135,in=0,in looseness=1] (4cm,3cm) to [out=180,in=45,out looseness=1](v4)
	;
	\draw[possibleedge]
	(v3)--(v5);
	\end{tikzpicture}
        }}
  \subcaptionbox*{$\mbox{type}=(0,1,c)$}{
	\scalebox{\scscaling}{
		\centering
		\begin{tikzpicture}
		
		\node[labeled] (v1) at (0,0) {$\vertex{1}$};
		\node[labeled] (v2) at (1,2) {$\vertex{2}$};
		\node[labeled] (v3) at (2,0) {$\vertex{3}$};
		\node[labeled] (v4) at (3,2) {$\vertex{4}$};
		\node[labeled] (v5) at (4,0) {$\vertex{5}$};
		\node[point] (v6) at (5,2) {};
		
		\draw[dteedge]
		(v2)--(v3)--(v4)--(v2)
		(v4)--(v5)
		;
		\draw[edge]
		(v1)--(v2)
		(v1)--(v3)
		;
		
		\draw[nonedge]
		(v4)--(v6)--(v5)
		;
		
		\draw[possibleedge]
		(v3)--(v5);
		
		\end{tikzpicture}
		\separator
		\begin{tikzpicture}
		\node[labeled] (v1) at (0,2) {$\vertex{1}$};
		\node[labeled] (v2) at (1,0) {$\vertex{2}$};
		\node[labeled] (v4) at (2,2) {$\vertex{4}$};
		\node[labeled] (va) at (3,0) {$\vertex{a}$};
		\node[labeled] (v3) at (4,2) {$\vertex{3}$};
		\node[labeled] (v5) at (5,0) {$\vertex{5}$};
		\node[point] (v6) at (6,2) {};
		
		\draw[dteedge]
		(v2)--(v4)--(va)--(v3)--(v4)
		(v2)--(va)
		(va)--(v5)
		;
		\draw[edge]
		(v1)--(v2)
		[-] (v1)  to [out=45,in=180,in looseness=1] (2cm,3cm) to [out=0,in=135,out looseness=1](v3)
		;
		\draw[possibleedge]
		(v3)--(v5);
		\draw[nonedge]
		(v5)--(v6)
		[-] (v6)  to [out=135,in=0,in looseness=1] (4cm,3cm) to [out=180,in=45,out looseness=1](v4)
		;
		\end{tikzpicture}
}}
  \subcaptionbox*{$\mbox{type}=(1,0,c)$}{
	\scalebox{\scscaling}{
		\centering
		\begin{tikzpicture}
		
		\node[point] (v1) at (0,0) {};
		\node[labeled] (v2) at (1,2) {$\vertex{2}$};
		\node[labeled] (v3) at (2,0) {$\vertex{3}$};
		\node[labeled] (v4) at (3,2) {$\vertex{4}$};
		\node[labeled] (v5) at (4,0) {$\vertex{5}$};
		\node[labeled] (v6) at (5,2) {$\vertex{6}$};
		
		\draw[dteedge]
		(v2)--(v3)--(v4)--(v2)
		(v4)--(v5)
		;
		\draw[edge]
		(v4)--(v6)--(v5)
		;
		
		\draw[possibleedge]
		(v3)--(v5);
		
		\draw[nonedge]
		(v2)--(v1)--(v3);
		
		\end{tikzpicture}
		\separator
		\begin{tikzpicture}
		\node[point] (v1) at (0,2) {};
		\node[labeled] (v2) at (1,0) {$\vertex{2}$};
		\node[labeled] (v4) at (2,2) {$\vertex{4}$};
		\node[labeled] (va) at (3,0) {$\vertex{a}$};
		\node[labeled] (v3) at (4,2) {$\vertex{3}$};
		\node[labeled] (v5) at (5,0) {$\vertex{5}$};
		\node[labeled] (v6) at (6,2) {$\vertex{6}$};
		
		\draw[dteedge]
		(v2)--(v4)--(va)--(v3)--(v4)
		(v2)--(va)
		(va)--(v5)
		;
		\draw[edge]
		(v5)--(v6)
		[-] (v6)  to [out=135,in=0,in looseness=1] (4cm,3cm) to [out=180,in=45,out looseness=1](v4)
		;
		\draw[possibleedge]
		(v3)--(v5);
		\draw[nonedge]
		(v1)--(v2)
		[-] (v1)  to [out=45,in=180,in looseness=1] (2cm,3cm) to [out=0,in=135,out looseness=1](v3)
		;
		\end{tikzpicture}
}}
  \subcaptionbox*{$\mbox{type}=(0,0,c)$}{
	\scalebox{\scscaling}{
		\centering
		\begin{tikzpicture}
		
		\node[point] (v1) at (0,0) {};
		\node[labeled] (v2) at (1,2) {$\vertex{2}$};
		\node[labeled] (v3) at (2,0) {$\vertex{3}$};
		\node[labeled] (v4) at (3,2) {$\vertex{4}$};
		\node[labeled] (v5) at (4,0) {$\vertex{5}$};
		\node[point] (v6) at (5,2) {};
		
		\draw[dteedge]
		(v2)--(v3)--(v4)--(v2)
		(v4)--(v5)
		;
		
		\draw[possibleedge]
		(v3)--(v5);
		
		\draw[nonedge]
		(v3)--(v1)--(v2)
		(v4)--(v6)--(v5)
		;
		
		\end{tikzpicture}
		\separator
		\begin{tikzpicture}
		\node[point] (v1) at (0,2) {};
		\node[labeled] (v2) at (1,0) {$\vertex{2}$};
		\node[labeled] (v4) at (2,2) {$\vertex{4}$};
		\node[labeled] (va) at (3,0) {$\vertex{a}$};
		\node[labeled] (v3) at (4,2) {$\vertex{3}$};
		\node[labeled] (v5) at (5,0) {$\vertex{5}$};
		\node[point] (v6) at (6,2) {};
		
		\draw[dteedge]
		(v2)--(v4)--(va)--(v3)--(v4)
		(v2)--(va)
		(va)--(v5)
		;
		\draw[nonedge]
		(v1)--(v2)
		(v5)--(v6)
		[-] (v1)  to [out=45,in=180,in looseness=1] (2cm,3cm) to [out=0,in=135,out looseness=1](v3)
		[-] (v6)  to [out=135,in=0,in looseness=1] (4cm,3cm) to [out=180,in=45,out looseness=1](v4)
		;
		\draw[possibleedge]
		(v3)--(v5);
		\end{tikzpicture}
}}

	\caption[Double triangle expansion types]{This figure shows a double triangle expansion of a triangle in various configurations. In each case, the triangle and edge subgraph $(\vertex{2},\vertex{4},\vertex{3},\vertex{5})$ was expanded to the double triangle and edge subgraph $(\vertex{2},\vertex{4},\vertex{a},\vertex{3},\vertex{5})$, both marked in \dtecolor, where $\vertex{a}$ is the vertex created in the double triangle expansion. Solid lines indicate present edges, and solid-dashed lines possible edges.  In all cases, $c=1$ if the solid-dashed line is present, and $c=0$ if it is absent. There are no additional triangles between the indicated vertices, though there may potentially be additional edges which do not cause new triangles such as an edge between $\vertex{1}$ and $\vertex{5}$ if $c=0$.  An unlabelled vertex is an arbitrary vertex in the graph, and, along with dashed lines, is used to denote the absence of common neighbors not already shown. For instance, in $\mbox{type}=(0,1,c)$, vertices $\vertex{4}$ and $\vertex{5}$ do not share any common neighbors other than possibly $\vertex{3}$. The effect of double triangle expansion on the chain vector is summarized in Table~\ref{tbl:dtrcveffect}.}
\label{fig:triangletypes}
\end{figure}

\begin{table}
	\caption[Double Triangle Reduction's Effect on the Chain Vector]{This table summarizes the different ways a single double triangle reduction can affect the chain vector of a pseudo-descendant. In this table, $l,m,n \in \mathbb{Z}_{\geq 0}$ and $c \in \{0,1\}$. We denote the changed part within brackets, and there can be any number of terms to the left of `$[$' or to the right of `$]$'.} 
	\begin{center}
		\begin{tabular}{|l|l|c|}
			\hline			
			\bf{Child vector part} & \bf{Parent vector part} & \bf{Type} \\ \hline
			$[m,3,n]$ & $[m+n+4]$ & $(1,1,1)$ \\ \hline
			$[m,3]$ & $[m+3]$ & $(0,1,1)$ \\ \hline
			$[3,n]$ & $[n+3]$ & $(1,0,1)$ \\ \hline
			$[m,2,0,n]$ & $[m+2,n+1]$ & $(1,1,0)$ \\ \hline
			$[m,2,0]$ & $[m+2,0]$ & $(0,1,0)$ \\ \hline
			$[2,0,n]$ & $[1,n+1]$ & $(1,0,0)$ \\ \hline
			$[l+2]$ & $[l+1]$ & $(0,0,c)$ \\ \hline
		\end{tabular}
	\end{center}
	\label{tbl:dtrcveffect}
\end{table}


\begin{lemma}
	Let $G$ be a $K_5$ descendant with a chain vector and at least $7$ vertices. Expanding a triangle in $G$ transforms the chain vector of $G$ from a parent vector to the corresponding child vector in Table~\ref{tbl:dtrcveffect}.
	\label{lem:triangletypes} 
\end{lemma}

Note that Corollary~\ref{result:chainvectordtrlist} will show that all $K_5$ descendants have chain vectors, but the hypothesis is necessary for the time being as this lemma is used in the proof of Corollary~\ref{result:chainvectordtrlist}.

\begin{proof}
	Assume the hypotheses.  Let the chain vector of $G$ be $(z_1,z_2,\cdots,z_k)$, where $k \geq 1$, and $z_i\geq 0$ for $1 \leq i \leq k$.  Let $H$ be a triangle-and-edge subgraph as in Definition~\ref{defn:triangletypes} of type $(a,b,c)$. By Lemma~\ref{lem zigzag shape} the triangles of $G$ are partitioned into maximal zigzag subgraphs and so all triangle-and-edge subgraphs fit into one of the cases of Lemma~\ref{lem zigzag shape} with no additional triangles between the vertices indicated in Figure~\ref{fig:triangletypes}.
	
	Consider the cases when $a+b\geq1$. If there is an edge $(\vertex{1},\vertex{2})$ that is part of a zigzag ($\vertex{1}\neq \vertex{3}, \vertex{4}$), then swap $\vertex{1}$ with the unnamed neighbour of $\vertex{2}$, if necessary, so that the unnamed neighbour is not a neighbour of $\vertex{3}$, and  let $m$ denote the number of triangles in that zigzag to the left of $(\vertex{1},\vertex{2})$ (where $\vertex{1}$ is to the right of the unnamed neighbour of $\vertex{2}$).  If there is no such $\vertex{1}$, let $m=0$.  Similarly, if there is an edge $(\vertex{5},\vertex{6})$ that is part of a zigzag ($\vertex{6}\neq \vertex{3}, \vertex{4}$), then swap $\vertex{6}$ with the unnamed neighbour of $\vertex{5}$, if necessary, so that the unnamed neighbour is not a neighbour of $\vertex{4}$, and let $n$ denote the number of triangles in that zigzag to the right of $(\vertex{5},\vertex{6})$ (where $\vertex{6}$ is to the left of the unnamed neighbour of $\vertex{5}$).  If there is no such $\vertex{6}$ let $n=0$.  The definition of $m$ and $n$ is illustrated in Figure~\ref{fig m and n}.

\begin{figure}
\[
        \includegraphics{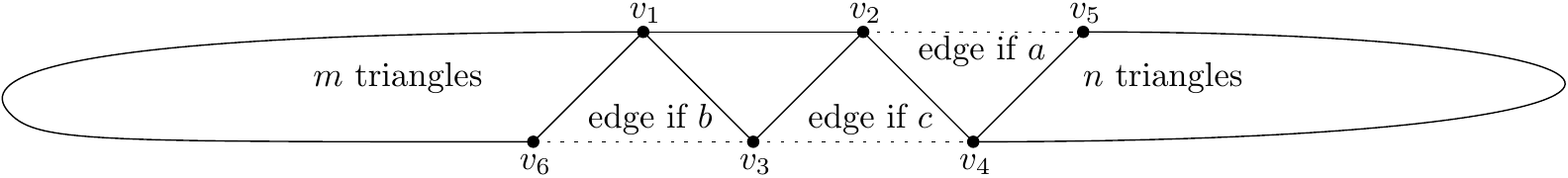}
\]
\caption{Schematic of the definition of $m$ and $n$ in the case $a+b\geq 1$.}\label{fig m and n}
\end{figure}
	
	Assume first that $c=1$. On the left-hand-side, we have the parent chain vector $(\dots, m+n+4, \dots)$, $(\dots, m+3, n, \dots)$ or $(\dots, m, n+3, \dots)$ if $a=b=1$, $a=0$ and $b=1$, or $a=1$ and $b=0$, respectively. After a single DTE, the child vector is $(\dots, m, 3, n, \dots)$, since the $m$ triangles on the left and the $n$ triangles on the right will be in separate zigzags from the middle $3$-triangle zigzag, and if $m=0$ or $n=0$ a gap, encoded by a $0$, will be created on either side, respectively.
	
	Assume now that $c=0$. On the left-hand-side, we have the parent chain vector $(\dots, m+2, n+1, \dots)$, $(\dots, m+2, 0, n, \dots)$ or $(\dots, m, 1, n+1, \dots)$ if $a=b=1$, $a=0$ and $b=1$, or $a=1$ and $b=0$, respectively. In all cases, the child chain vector is $(\dots, m, 2, 0, n, \dots)$, as was argued for when $c=1$.
	
	Finally, assume that $a=b=0$. Let $l\geq 0$ be the number of triangles in the maximal zigzag piece containing the triangle $(\vertex{2},\vertex{3},\vertex{4})$.  Note that unlike in the other cases, we can have the triangles $(\vertex{2}, \vertex{4}, v)$ or $(\vertex{3}, \vertex{5}, u)$ for some $u,v \notin \{\vertex{2},\vertex{3},\vertex{4},\vertex{5}\}$ with $u \neq v$, and in fact these are the only directions that the zigzag containing $(\vertex{2}, \vertex{3}, \vertex{4})$ can go in. 
        Whether $c=0$ or $c=1$, a single DTE will transform the parent vector $(\dots, l+1,\dots)$ to the child vector $(\dots, l+2, \dots)$.
\end{proof}

Since double triangle reduction is the inverse of double triangle expansion, we conclude the following corollary from Lemma~\ref{lem:triangletypes}.

\begin{cor}
	Let $G$ be a $K_5$ descendant with at least $8$ vertices. Reducing a double triangle in $G$ transforms the chain vector of $G$ from a child vector to the corresponding parent vector in Table~\ref{tbl:dtrcveffect}.
	\label{result:chainvectordtrlist}
\end{cor}

By induction with Lemma~\ref{lem:triangletypes}, beginning with the unique $K_5$ descendant with $7$ vertices (which is a $1$-zigzag) we also obtain the following corollary.

\begin{cor}\label{lem K5 desc is pseudo desc} 
        Let $G$ be a $K_5$ descendant with at least 7 vertices, then $G$ has a chain vector. That is, either $G$ is an $n$-zigzag or all triangles of $G$ are in open chains.
\end{cor}

Note that Corollary~\ref{lem K5 desc is pseudo desc} says more than Lemma~\ref{lem zigzag shape} as it means there can be at most one closed zigzag chain and in that case there are no open zigzag chains.

We note that, for a general $4$-regular graph $G$ with all triangles in open chains, Corollary~\ref{result:chainvectordtrlist} does not necessarily hold. For example, referring to Figure~\ref{fig:triangletypes}, a double triangle in $G$ might have the vertices as in type $(1,1,1)$, with the additional edges $(\vertex{2},v)$ and $(\vertex{4},v)$. Upon reducing the double triangle, a triple triangle $K_{3,1,1}$ is created. This is not possible if $G$ is a $K_5$ descendant with at least $8$ vertices, as the only $K_5$ descendant with a triple triangle is $K_5$ itself (Lemma~\ref{lem zigzag shape}).


The converse of Corollary~\ref{lem K5 desc is pseudo desc} does not hold. For example, any $4$-regular graph with no triangles has a chain vector (namely, $(0)$), but is not a $K_5$ descendant. See \cite[Figure 4.9]{LaradjiDTDoK5} for a less trivial example.

By inspection of the outcome of a DTE of each triangle type and checking the small cases directly we obtain the following lemma.
\begin{lemma}\label{lem:levelchange}
	Let $G$ be a $K_5$-descendant with at least $7$ vertices, and let $H$ be a child of $G$. Then,
	\begin{itemize}
		\item $\emph{tri}(H)=\emph{tri}(G)+c$, where $c \in \{-1,0,1\}$
		\item $\mbox{L}(H)=\mbox{L}(G)+c$, where $c \in \{0,1,2\}$.
	\end{itemize}
\end{lemma}

Through Lemma~\ref{lem:levelchange} and the definition of the chain vector, we can also conclude the following lemma.
\begin{lemma}
	All $K_5$-descendants of order $\geq 9$ that are not $1$-zigzags can be double triangle reduced to the unique $K_5$-descendant of level $2$ with $8$ vertices; this $K_5$ descendant has chain vector $(3,3)$.
	\label{lem:anc33}
\end{lemma}

\begin{proof}
        Let $G$ be as described in the statement and let $L$ be its level.  By Proposition~\ref{prop:level0} $G$ not being a 1-zigzag means that $L\neq 0$.  By Proposition~\ref{prop:level1} $L\neq 1$, so $L\geq 2$.  By Lemma~\ref{lem:levelchange} double triangle reductions decrease $L$ by $1$ or $2$ or leave it unchanged.  Therefore, some ancestor of $G$ has level $2$.  By Proposition~\ref{prop:level2} this ancestor is a completed primitive $2$-zigzag which by appropriate choice of double triangle reductions can be reduced to the completed primitive $2$-zigzag with chain vector $(3,3)$.
\end{proof}

Through the definition of the chain vector, we have the following result.
\begin{lemma}
	Let $G$ be a $K_5$ descendant with at least $7$ vertices, and suppose that $(z_1, \dots, z_k)$ is a chain vector of $G$. Then,
	\begin{align*}
		\sum{z_i(G)}=\mbox{tri}(G).
	\end{align*}
	\label{lem:triG is sumCV}
\end{lemma}

As was argued in the proof of Proposition~\ref{prop:level2}, a $2$-zigzag $G$ must have at least $3$ triangles in each zigzag, as otherwise additional triangles are created and $G$ is no longer a $2$-zigzag. On the other hand, recall from Tables~\ref{tab L=3 not small} and~\ref{tab L=3 small} in the proof of Proposition~\ref{prop:level3} that $3$-zigzags that are $K_5$ descendants must have at least $3$ triangles in each of at least two zigzags. Hence, we arrive at the following lemma.
\begin{lemma}
	Let $G$ be a $K_5$ descendant and $m \geq 2$. If $G$ has a chain vector, then it cannot be one of the following:
	\begin{align*}
		(m,1),(m,1,1),(m,2).
	\end{align*}
	\label{lem:unrealizable cv for k5d}
\end{lemma}

Additionally, we can observe some chain vectors where only the $(0,0,c)$ DTR is possible
\begin{lemma}\label{lem 00c observe}\mbox{}
\begin{itemize}
        \item The only transformation of the chain vector $(n)$ by DTRs is to $(n-1)$, and so in particular no chain of DTRs on $(n)$ can reach $(3,3)$.
        \item  If a chain vector has no entries equal to $2$ or $3$ then only the type $(0,0,c)$ DTR is possible, and so in particular for chain vectors with exactly one entry greater than $3$ and all other entries $0$ or $1$, the only possible DTRs lead to the chain vector with the large entry replaced by $3$ and the other entries unchanged.  From there other reductions may be possible.
\end{itemize}
\end{lemma}

\subsection{Minimum triangle count}

\noindent We now proceed to prove the main result of the section.

%

\begin{thm}
	If $G$ is a $K_5$-descendant, then $\mbox{tri}(G)\geq 4$.
	\label{thm:tri4}
\end{thm}
\begin{proof}
	Assume, for a contradiction, that $G$ is a $K_5$-descendant with $\mbox{tri}(G)\leq3$. If $\mbox{order}(G)\leq8$, then $G$ is one of $5$ possible graphs (see Figure~\ref{fig:K5familytree}), all of which have $\mbox{tri}(G)\geq 6$. If $G$ is a $1$-zigzag of order $\geq 7$, then $\mbox{tri}(G)=\mbox{order}(G)\geq 7$. We assume then that $\mbox{order}(G)\geq9$ and that $G$ is not a $1$-zigzag. By Lemma~\ref{lem K5 desc is pseudo desc}, $G$ has a chain vector.
	
	Since a DTE always leaves a double triangle, at least one of the zigzags in $G$ has $\geq 2$ triangles. By Lemma~\ref{lem:triG is sumCV} and Lemma~\ref{lem:unrealizable cv for k5d}, the chain vector of $G$ must be one of the following:
	\begin{align*}
	(2,0),(2,1,0),(2,0,1,0),(3,0).
	\end{align*}

	Through Theorem \ref{thm:ancestor} and Lemma \ref{lem:anc33}, we can start from $G$ and arrive at the chain vector $(3,3)$ through double triangle reductions exclusively (that is, without double triangle expansions). 
	
        Since $G$ is a $K_5$ descendant with at least $9$ vertices, by Corollary~\ref{result:chainvectordtrlist}, the list of all possible transformations of the chain vector is in Table \ref{tbl:dtrcveffect}.
%
%
In what follows, we want to show that starting with any of the four possible chain vectors for $G$, we either
\begin{itemize}
        \item reach a chain vector unrealizable for a $K_5$ descendant by Lemma~\ref{lem:unrealizable cv for k5d},
        \item reach a chain vector that does not have any elements $\geq2$ and hence cannot belong to a $K_5$ descendant, such chain vectors will be called \emph{invalid},
        \item  reach a chain vector where the only possible reductions are as described in Lemma~\ref{lem 00c observe}, and where for the second point of Lemma~\ref{lem 00c observe} the result of such reductions puts us into one of the other points itemized here, or
        \item get stuck in a loop.
\end{itemize}
That is, we cannot reach the chain vector $(3,3)$. 
	
	For readability, the proof is organized into Tables~\ref{tri4:tbl:20}--\ref{tri4:tbl:31010}.  The tables should be interpreted as follows.  $CV^{(i)}$ refers to the chain vector after $i$ double triangle reductions.  The column labelled ``Outcome'' justifies why this new chain vector cannot reach $(3,3)$ by appealing to the points above or indicating which table this chain vector is addressed in.  The starting chain vectors are listed in the first column in whichever equivalent form clearly illustrates the DTR we are applying.  Note that this means that there are often consecutive $0$s in these chain vectors; it is important to consider such possibilities in order to enumerate every eventuality.

		\newcommand{\onlyones}{Invalid}
	\newcommand{\notkfived}{Lemma~\ref{lem:unrealizable cv for k5d}}
        \newcommand{\observe}{Lemma~\ref{lem 00c observe}}
	
	\newcommand{\ncols}{1}
	\renewcommand{\ncols}{5}
	\begin{table}[ht]
		\caption{$CV^{(0)}=(2,0)$.}
		\begin{center}
			\begin{tabular}{|*{\ncols}{c|}}
				\hline
				$CV^{(0)}$ & \multicolumn{2}{c|}{$CV^{(1)}$} & DTR type & Outcome\\ \hline
				
				$(0,2,0,0)$ &
				$[0,2,0,0]\rightarrow[2,1]$ & $(2,1)$ & $(1,1,0)$ & \notkfived \\ \hline
				
				$(0,2,0,0,0)$ &
				$[0,2,0,0]\rightarrow[2,1]$ & $(2,1,0)$ & $(1,1,0)$ & Table~\ref{tri4:tbl:210} \\ \hline
				
				$(0,2,0,0)$ &
				$[0,2,0]\rightarrow[2,0]$ & $(2,0)$ & $(0,1,0)$ & Table~\ref{tri4:tbl:20} \\ \hline
				
				$(2,0,0)$ &
				$[2,0,0]\rightarrow[1,1]$ & $(1,1)$ & $(1,0,0)$ & \onlyones \\ \hline
				
				$(2,0,0,0)$ &
				$[2,0,0]\rightarrow[1,1]$ & $(1,1,0)$ & $(1,0,0)$ & \onlyones \\ \hline
				
				$(2,0)$ &
				$[2]\rightarrow[1]$ & $(1,0)$ & $(0,0,c)$ & \onlyones \\ \hline		
			\end{tabular}
		\end{center}
		\label{tri4:tbl:20}
	\end{table}		

	\renewcommand{\ncols}{8}
	\begin{table}[ht]
		\caption{$CV^{(0)}=(2,1,0)$.}
		\begin{center}
			\begin{tabular}{|*{\ncols}{c|}}
				\hline
				
				$CV^{(0)}$ & \multicolumn{2}{c|}{$CV^{(1)}$} & DTR type & Outcome\\ \hline
				
				$(1,2,0,0)$ &
				$[1,2,0,0]\rightarrow[3,1]$ & $(3,1)$ & $(1,1,0)$ & \notkfived \\ \hline
				
				$(1,2,0,0,0)$ &
				$[1,2,0,0]\rightarrow[3,1]$ & $(3,1,0)$ & $(1,1,0)$ & Table~\ref{tri4:tbl:310} \\ \hline
				
				$(1,2,0)$ &
				$[1,2,0]\rightarrow[3,0]$ & $(3,0)$ & $(0,1,0)$ & Table~\ref{tri4:tbl:30} \\ \hline
				
				$(2,0,0,1)$ &
				$[2,0,0]\rightarrow[1,1]$ & $(1,1,1)$ & $(1,0,0)$ & \onlyones \\ \hline
				
				$(2,0,0,0,1)$ &
				$[2,0,0]\rightarrow[1,1]$ & $(1,1,0,1)$ & $(1,0,0)$ & \onlyones \\ \hline
				
				$(2,0,1)$ &				
				$[2,0,1]\rightarrow[1,2]$ & $(1,2)$ & $(1,0,0)$ & \notkfived \\ \hline
				
				$(2,1,0)$ &
				$[2]\rightarrow[1]$ & $(1,1,0)$ & $(0,0,c)$ & \onlyones \\\hline
	
			\end{tabular}
		\end{center}
	\label{tri4:tbl:210}
	\end{table}		

	\renewcommand{\ncols}{5}
	\begin{table}[ht]
		\caption{$CV^{(0)}=(2,0,1,0)$.} 
		\begin{center}
			\begin{tabular}{|*{\ncols}{c|}}
				\hline
								
				$CV^{(0)}$ & 
				\multicolumn{2}{c|}{$CV^{(1)}$} & DTR type & Outcome\\ \hline
				
				$(0,2,0,0,1)$ & $[0,2,0,0]\rightarrow [2,1]$ & $(2,1,1)$ & $(1,1,0)$ & \notkfived \\\hline
				$(0,2,0,0,1,0)$ & $[0,2,0,0]\rightarrow [2,1]$ & $(2,1,1,0)$ & $(1,1,0)$ & Table~\ref{tri4:tbl:2110} \\\hline
                                $(0,2,0,0,0,1)$ & $[0,2,0,0]\rightarrow [2,1]$ & $(2,1,0,1)$ & $(1,1,0)$ & Table~\ref{tri4:tbl:1310} \\\hline
                                $(0,2,0,0,0,1,0)$ & $[0,2,0,0]\rightarrow [2,1]$ & $(2,1,0,1,0)$ & $(1,1,0)$ & Table~\ref{tri4:tbl:31010} \\\hline
				
				$(0,2,0,1)$ & $[0,2,0,1]\rightarrow [2,2]$ & $(2,2)$ & $(1,1,0)$ & \notkfived \\ \hline
                                $(0,2,0,1,0)$ & $[0,2,0,1]\rightarrow [2,2]$ & $(2,2,0)$ & $(1,1,0)$ & Table~\ref{tri4:tbl:220} \\ \hline

                                $(2,0,0,1,0)$ & $[2,0,0]\rightarrow[1,1]$ & $(1,1,1,0)$ & $(1,0,0)$ & \onlyones \\\hline
                                $(2,0,0,0,1,0)$ & $[2,0,0]\rightarrow[1,1]$ & $(1,1,0,1,0)$ & $(1,0,0)$ & \onlyones \\\hline
				$(2,0,1,0)$ & $[2,0,1]\rightarrow[1,2]$ & $(1,2,0)$ & $(1,0,0)$ & Table~\ref{tri4:tbl:210} \\\hline

				$(0,2,0,0,1)$ & $[0,2,0]\rightarrow[2,0]$ & $(1,2,0)$ & $(0,1,0)$ & Table~\ref{tri4:tbl:210} \\\hline
				$(0,2,0,1)$ & $[0,2,0]\rightarrow[2,0]$ & $(1,2)$ & $(0,1,0)$ & \notkfived \\\hline
				
				$(2,0,1,0)$ & $[2]\rightarrow[1]$ & $(1,0,1,0)$ & $(0,0,c)$ & \onlyones \\\hline			
			\end{tabular}
		\end{center}
	\label{tri4:tbl:2010}
	\end{table}
        
	\renewcommand{\ncols}{5}
	\begin{table}[ht]
		\caption{$CV^{(0)}=(3,0)$.}
		\begin{center}
			\begin{tabular}{|*{\ncols}{c|}}
				\hline
				$CV^{(0)}$ & \multicolumn{2}{c|}{$CV^{(1)}$} & DTR type & Outcome\\ \hline
				$(0,3,0)$ & $[0,3,0]\rightarrow[3]$ & $(4)$ & $(1,1,1)$  & \observe \\ \hline
				$(0,3,0,0)$ & $[0,3,0]\rightarrow[3]$ & $(4,0)$ & $(1,1,1)$ & \observe\ and Table~\ref{tri4:tbl:30} \\\hline
				$(3,0)$ & $[3,0]\rightarrow[3]$ & $(3)$ & $(1,0,1)$ & \observe \\ \hline
				$(3,0,0)$ & $[3,0]\rightarrow[3]$ & $(3,0)$ & $(1,0,1)$ & Table~\ref{tri4:tbl:30} \\ \hline
				$(3,0)$ & $[3]\rightarrow[2]$ & $(2,0)$ & $(0,0,c)$ & Table~\ref{tri4:tbl:20} \\\hline
			\end{tabular}
		\end{center}
		\label{tri4:tbl:30}
	\end{table}

        	\renewcommand{\ncols}{5}
	\begin{table}[ht]
		\caption{$CV^{(0)}=(2,2,0)$.}
		\begin{center}
			\begin{tabular}{|*{\ncols}{c|}}
				\hline
				$CV^{(0)}$ & \multicolumn{2}{c|}{$CV^{(1)}$} & DTR type & Outcome\\ \hline
                                $(2,2,0,0)$ & $[2,2,0,0]\rightarrow [4,1]$ & $(4,1)$ & $(1,1,0)$ & \notkfived\\ \hline				
				$(2,2,0,0,0)$ & $[2,2,0,0]\rightarrow [4,1,0]$ & $(4,1,0)$ & $(1,1,0)$ & \observe\ and Table~\ref{tri4:tbl:310}\\ \hline
				$(2,2,0)$ & $[2,2,0]\rightarrow[4,0]$ & $(4,0)$ & $(0,1,0)$ & \observe\ and Table~\ref{tri4:tbl:30}\\ \hline			
				$(2,0,2)$ & $[2,0,2]\rightarrow[1,3]$ & $(3,1)$ & $(1,0,0)$ & \notkfived \\ \hline
				$(2,0,0,2)$ & $[2,0,0]\rightarrow[1,1]$ & $(1,1,2)$ & $(1,0,0)$ & \notkfived \\ \hline
                                $(2,0,0,0,2)$ & $[2,0,0]\rightarrow[1,1]$ & $(1,1,0,2)$ & $(1,0,0)$ & Table~\ref{tri4:tbl:2110} \\ \hline
				$(2,2,0)$ & $[2]\rightarrow[1]$ & $(2,1,0)$ & $(0,0,c)$ & Table~\ref{tri4:tbl:210}\\ \hline
				
			\end{tabular}
		\end{center}
		\label{tri4:tbl:220}
	\end{table}

		\renewcommand{\ncols}{5}
	\begin{table}[ht]
		\caption{$CV^{(0)} = (3,1,0)$}
		\begin{center}
			\begin{tabular}{|*{\ncols}{c|}}
				\hline
				
				$CV^{(0)}$ & \multicolumn{2}{c|}{$CV^{(1)}$} & DTR type & Outcome\\ \hline
				
				$(0,3,1)$ &
				$[0,3,1]\rightarrow[4]$ & $(4)$ & $(1,1,1)$ & \observe \\ \hline
				
				$(0,3,1,0)$ &
				$[0,3,1]\rightarrow[4]$ & $(4,0)$ & $(1,1,1)$ & \observe\ and Table~\ref{tri4:tbl:30}\\ \hline
				$(1,3,0)$ &
				$[1,3,0]\rightarrow[4]$ & $(4)$ & $(1,1,1)$ & \observe \\ \hline
				
				$(1,3,0,0)$ &
				$[1,3,0]\rightarrow[4]$ & $(4,0)$ & $(1,1,1)$ & \observe\ and Table~\ref{tri4:tbl:30}\\ \hline
				
				$(3,1,0)$ &
				$[3,1]\rightarrow[4]$ & $(4,0)$ & $(1,0,1)$ & \observe\ and Table~\ref{tri4:tbl:30}\\ \hline

                                $(3,0,1)$ &
                                $[3,0]\rightarrow[3]$ & $(3,1)$ & $(1,0,1)$ & \notkfived \\ \hline
                                $(3,0,0,1)$ &
                                $[3,0]\rightarrow[3]$ & $(3,0,1)$ & $(1,0,1)$ & Table~\ref{tri4:tbl:310} \\ \hline

                                $(0,3,1)$ &
				$[0,3]\rightarrow[3]$ & $(3,1)$ & $(0,1,1)$ & \notkfived\\ \hline
                                
				$(0,3,1,0)$ &
				$[0,3]\rightarrow[3]$ & $(3,1,0)$ & $(0,1,1)$ & Table~\ref{tri4:tbl:310}\\ \hline

                                $(1,3,0)$ &
				$[1,3]\rightarrow[4]$ & $(4,0)$ & $(0,1,1)$ & \observe\ and Table~\ref{tri4:tbl:30} \\ \hline
                                
				$(3,1,0)$ &
				$[3]\rightarrow[2]$ & $(2,1,0)$ & $(0,0,c)$ & Table~\ref{tri4:tbl:210}\\ \hline

          			\end{tabular}
		\end{center}
		\label{tri4:tbl:310}
	\end{table}                      

        	\renewcommand{\ncols}{5}
	\begin{table}[ht]
		\caption{$CV^{(0)}=(3,0,1,0)$.}
		\begin{center}
			\begin{tabular}{|*{\ncols}{c|}}
				\hline
				
				$CV^{(0)}$ & \multicolumn{2}{c|}{$CV^{(1)}$} & DTR type & Outcome\\ \hline
				$(0,3,0,0,1,0)$ & $[0,3,0]\rightarrow[4]$ & $(4,0,1,0)$ & $(1,1,1)$ & Table~\ref{tri4:tbl:31010}
				\\ \hline
				$(0,3,0,1,0)$ & $[0,3,0]\rightarrow[4]$ & $(4,1,0)$ & $(1,1,1)$ & \observe\ and Table~\ref{tri4:tbl:310}
				\\ \hline
				$(0,3,0,0,1)$ & $[0,3,0]\rightarrow[4]$ & $(4,0,1)$ & $(1,1,1)$ & Table~\ref{tri4:tbl:2010}
				\\ \hline
				$(0,3,0,1)$ & $[0,3,0]\rightarrow[4]$ & $(4,1)$ & $(1,1,1)$ & \notkfived
				\\ \hline
				$(0,3,0,1)$ & $[0,3]\rightarrow[3]$ & $(3,0,1)$ & $(0,1,1)$ & Table~\ref{tri4:tbl:310}
				\\ \hline
				$(0,3,0,1,0)$ & $[0,3]\rightarrow[3]$ & $(3,0,1,0)$ & $(0,1,1)$ & Table~\ref{tri4:tbl:31010}
				\\ \hline
				$(3,0,1,0)$ & $[3,0]\rightarrow[3]$ & $(3,1,0)$ & $(1,0,1)$ & Table~\ref{tri4:tbl:310}
				\\ \hline
				$(3,0,0,1,0)$ & $[3,0]\rightarrow[3]$ & $(3,0,1,0)$ & $(1,0,1)$ & Table~\ref{tri4:tbl:31010}
				\\ \hline                                
				$(3,0,1,0)$ & $[3]\rightarrow[2]$ & $(2,0,1,0)$ & $(0,0,c)$ & Table~\ref{tri4:tbl:2010}
				\\ \hline

			\end{tabular}
		\end{center}
		\label{tri4:tbl:3010}
	\end{table}

        \begin{table}[ht]
		\caption{$CV^{(0)} = (3,2,0)$}
		\begin{center}
			\begin{tabular}{|*{\ncols}{c|}}
				\hline
				
				$CV^{(0)}$ & \multicolumn{2}{c|}{$CV^{(1)}$} & DTR type & Outcome\\ \hline				
				$(0,3,2)$ & 
				$[0,3,2]\rightarrow [5]$ & $(5)$ & $(1,1,1)$ & \observe \\\hline
				
				$(0,3,2,0)$ & 
				$[0,3,2]\rightarrow [5]$ & $(5,0)$ & $(1,1,1)$ & \observe\ and Table~\ref{tri4:tbl:30}\\\hline

                                $(2,3,0)$ & 
				$[2,3,0]\rightarrow [5]$ & $(5)$ & $(1,1,1)$ & \observe \\\hline
				
				$(2,3,0,0)$ & 
				$[2,3,0]\rightarrow [5]$ & $(5,0)$ & $(1,1,1)$ & \observe\ and Table~\ref{tri4:tbl:30}\\\hline
				
				$(0,3,2)$ & 
				$[0,3]\rightarrow [3]$ & $(3,2)$ & $(0,1,1)$ & Lemma~\ref{lem:unrealizable cv for k5d}\\\hline

				$(0,3,2,0)$ & 
				$[0,3]\rightarrow [3]$ & $(3,2,0)$ & $(0,1,1)$ & Table~\ref{tri4:tbl:320}\\\hline

                                $(2,3,0)$ & 
				$[2,3]\rightarrow [5]$ & $(5,0)$ & $(0,1,1)$ & \observe\ and Table~\ref{tri4:tbl:30}\\\hline

				$(3,2,0)$ & 
				$[3,2]\rightarrow [5]$ & $(0,5)$ & $(1,0,1)$ & \observe\ and Table~\ref{tri4:tbl:30}\\\hline

                                $(3,0,2)$ & 
				$[3,0]\rightarrow [3]$ & $(3,2)$ & $(1,0,1)$ &  Lemma~\ref{lem:unrealizable cv for k5d}\\\hline

                                $(3,0,0,2)$ & 
				$[3,0]\rightarrow [3]$ & $(3,0,2)$ & $(1,0,1)$ & Table~\ref{tri4:tbl:320}\\\hline

				$(3,2,0,0)$ &
				$[3,2,0,0]\rightarrow [5,1]$ & $(5,1)$ & $(1,1,0)$ & Lemma~\ref{lem:unrealizable cv for k5d}\\\hline
				
				$(3,2,0,0,0)$ &
				$[3,2,0,0]\rightarrow [5,1]$ & $(5,1,0)$ & $(1,1,0)$ & \observe\ and Table~\ref{tri4:tbl:310}\\\hline
				
				$(3,2,0)$ &
				$[3,2,0]\rightarrow[5,0]$ & $(5,0)$ & $(0,1,0)$ & \observe\ and Table~\ref{tri4:tbl:30}\\\hline
				
				$(2,0,3)$ &
				$[2,0,3]\rightarrow[1,4]$ & $(4,1)$ & $(1,0,0)$ & \observe\ and Lemma~\ref{lem:unrealizable cv for k5d}\\\hline
				
				$(3,2,0,0)$ &
				$[2,0,0]\rightarrow[1,1]$ & $(3,1,1)$ & $(1,0,0)$ & Lemma~\ref{lem:unrealizable cv for k5d}\\\hline
				
				$(3,2,0,0,0)$ &
				$[2,0,0]\rightarrow[1,1]$ & $(3,1,1,0)$ & $(1,0,0)$ & Table~\ref{tri4:tbl:3110}\\\hline	
                                $(3,2,0)$ &
                                
$[3] \rightarrow [2]$ & $(2,2,0)$ & $(0,0,c)$ &Table~\ref{tri4:tbl:220}\\\hline 			
	
				$(3,2,0)$ & $[2] \rightarrow [1]$ & $(3,1,0)$ &$(0,0,c)$ & Table~\ref{tri4:tbl:310} \\\hline					
				
			\end{tabular}
		\end{center}
		\label{tri4:tbl:320}
	\end{table}

	\renewcommand{\ncols}{5}
	\begin{table}[ht]
		\caption{$CV^{(0)} = (3,1,1,0)$.}
		\begin{center}
			\begin{tabular}{|*{\ncols}{c|}}
				\hline
				
                                $CV^{(0)}$ & 
				\multicolumn{2}{c|}{$CV^{(1)}$} & DTR type & Outcome\\ \hline
				
				$(0,3,1,1)$ & $[0,3,1]\rightarrow[5]$ & $(5,1)$ & $(1,1,1)$ & \notkfived\\ \hline				
				$(0,3,1,1,0)$ & 
				$[0,3,1]\rightarrow[5]$ & $(5,1,0)$ & $(1,1,1)$ & \observe\ and Table~\ref{tri4:tbl:310}\\ \hline

                                $(1,3,0,1)$ & $[1,3,0]\rightarrow[5]$ & $(5,1)$ & $(1,1,1)$ & \notkfived\\ \hline				
				$(1,3,0,0,1)$ & 
				$[1,3,0]\rightarrow[5]$ & $(5,0,1)$ & $(1,1,1)$ & \observe\ and Table~\ref{tri4:tbl:310}\\ \hline
				
				$(3,1,1,0)$ & $[3,1]\rightarrow[4]$ & $(4,1,0)$ & $(1,0,1)$ & \observe\ and Table~\ref{tri4:tbl:310}\\ \hline
				$(3,0,1,1)$ & $[3,0]\rightarrow[3]$ & $(3,1,1)$ & $(1,0,1)$ & \notkfived\\ \hline
                                $(3,0,0,1,1)$ & $[3,0]\rightarrow[3]$ & $(3,0,1,1)$ & $(1,0,1)$ & Table~\ref{tri4:tbl:3110}\\ \hline
                                
				$(0,3,1,1)$ & $[0,3]\rightarrow[3]$ & $(3,1,1)$ & $(0,1,1)$ & \notkfived\\ \hline
				$(0,3,1,1,0)$ & $[0,3]\rightarrow[3]$ & $(3,1,1,0)$ & $(0,1,1)$ & Table~\ref{tri4:tbl:3110}\\ \hline
                                $(1,3,0,1)$ & $[1,3]\rightarrow[4]$ & $(4,0,1)$ & $(0,1,1)$ & \observe\ and Table~\ref{tri4:tbl:310}\\ \hline

                                $(3,1,1,0)$ & $[3]\rightarrow[2]$ & $(2,1,1,0)$ & $(0,0,c)$ &Table~\ref{tri4:tbl:2110}\\ \hline 

			\end{tabular}
		\end{center}
		\label{tri4:tbl:3110}
	\end{table}

	\renewcommand{\ncols}{5}
	\begin{table}[ht]
		\caption{$CV^{(0)} =(2,1,1,0)$.}
		\begin{center}
			\begin{tabular}{|*{\ncols}{c|}}
				\hline
				
                                $CV^{(0)}$ & 
				\multicolumn{2}{c|}{$CV^{(1)}$} & DTR type & Outcome\\ \hline

                                $(1,2,0,0,1)$ & 
				$[1,2,0,0]\rightarrow[3,1]$ & $(3,1,1)$ & $(1,1,0)$ & \notkfived \\ \hline
				$(1,2,0,0,0,1)$ & $[1,2,0,0]\rightarrow[3,1]$ & $(1,3,1,0)$ & $(1,1,0)$ & Table~\ref{tri4:tbl:1310} \\ \hline
				
				$(1,2,0,1)$ & $[1,2,0,1]\rightarrow[3,2]$ & $(3,2)$ & $(1,1,0)$ & \notkfived \\ \hline
				
				$(2,0,1,1)$ & $[2,0,1]\rightarrow[1,2]$ & $(1,2,1)$ & $(1,0,0)$ & \notkfived \\ \hline
				$(2,0,0,1,1)$ & $[2,0,0]\rightarrow[1,1]$ & $(1,1,1,1)$ & $(1,0,0)$ & \onlyones \\ \hline
				$(2,0,0,0,1,1)$ & $[2,0,0]\rightarrow[1,1]$ & $(1,1,0,1,1)$ & $(1,0,0)$ & \onlyones \\ \hline
				
                                $(1,2,0,1)$ & $[1,2,0]\rightarrow[3,0]$ & $(3,1,0)$ & $(0,1,0)$ & Table~\ref{tri4:tbl:210} \\ \hline

                                $(2,1,1,0)$ & $[2]\rightarrow [1]$ & $(1,1,1,0)$ & $(0,0,c)$ & \onlyones \\ \hline				
%
%
				
			\end{tabular}
		\end{center}
		\label{tri4:tbl:2110}
	\end{table}	

	\renewcommand{\ncols}{5}
	\begin{table}[ht]
		\caption{$CV^{(0)}\in\{(1,3,1,0), (1,2,1,0)\}$.}
		\begin{center}
			\begin{tabular}{|*{\ncols}{c|}}
				\hline
				
				$CV^{(0)}$ & 
				\multicolumn{2}{c|}{$CV^{(1)}$} & DTR type & Outcome\\ \hline
				
				$(1,3,1,0)$ &
				$[1,3,1]\rightarrow[6]$ & $(6,0)$ & $(1,1,1)$ & \observe\ and Table~\ref{tri4:tbl:30} \\ \hline

                                $(1,3,1,0)$ &
				$[1,3]\rightarrow[4]$ & $(4,1,0)$ & $(0,1,1)$ & \observe\ and Table~\ref{tri4:tbl:310} \\ \hline

                                $(3,1,0,1)$ &
				$[3,1]\rightarrow[4]$ & $(4,0,1)$ & $(1,0,1)$ & \observe\ and Table~\ref{tri4:tbl:310} \\ \hline

                                $(1,3,1,0)$ &
				$[3]\rightarrow[2]$ & $(1,2,1,0)$ & $(0,0,c)$ & Table~\ref{tri4:tbl:1310} \\ \hline

                                $(1,2,1,0)$ &
				$[2]\rightarrow[1]$ & $(1,1,1,0)$ & $(0,0,c)$ & \onlyones \\ \hline		
			\end{tabular}
		\end{center}
		\label{tri4:tbl:1310}
	\end{table}

	\renewcommand{\ncols}{5}
	\begin{table}[ht]
		\caption{$CV^{(0)}\in\{(3,1,0,1,0), (2,1,0,1,0)\}$.}
		\begin{center}
			\begin{tabular}{|*{\ncols}{c|}}
				\hline
				
				$CV^{(0)}$ & \multicolumn{2}{c|}{$CV^{(1)}$} & DTR type & Outcome\\ \hline
				
				$(1,3,0,1,0)$ & $[1,3,0]\rightarrow[5]$ & $(5,1,0)$ & $(1,1,1)$ & \observe\ and Table~\ref{tri4:tbl:2010}
				\\ \hline
				
				$(1,3,0,0,1,0)$ & $[1,3,0]\rightarrow[5]$ & $(5,0,1,0)$ & $(1,1,1)$ & \observe\ and Table~\ref{tri4:tbl:3010} \\ \hline
                                
				$(0,3,1,0,1)$ & $[0,3,1]\rightarrow[5]$ & $(5,0,1)$ & $(1,1,1)$ & \observe\ and Table~\ref{tri4:tbl:2010}
				\\ \hline
				
				$(0,3,1,0,1,0)$ & $[0,3,1]\rightarrow[5]$ & $(5,0,1,0)$ & $(1,1,1)$ & \observe\ and Table~\ref{tri4:tbl:3010} \\ \hline				
				$(3,1,0,1,0)$ & $[3,1]\rightarrow[4]$ & $(4,0,1,0)$ & $(1,0,1)$ & \observe\ and Table~\ref{tri4:tbl:31010} \\\hline
				$(3,0,1,0,1)$ & $[3,0]\rightarrow[3]$ & $(3,1,0,1)$ & $(1,0,1)$ & Table~\ref{tri4:tbl:1310} \\\hline

				$(0,3,1,0,1)$ & $[0,3]\rightarrow[3]$ & $(3,1,0,1)$ & $(0,1,1)$ & Table~\ref{tri4:tbl:1310} \\\hline
				$(0,3,1,0,1,0)$ & $[0,3]\rightarrow[3]$ & $(3,1,0,1,0)$ & $(0,1,1)$ & Table~\ref{tri4:tbl:31010} \\\hline
				$(1,3,0,1,0)$ & $[1,3]\rightarrow[4]$ & $(4,0,1,0)$ & $(0,1,1)$ & \observe\ and Table~\ref{tri4:tbl:3010} \\\hline
				$(3,1,0,1,0)$ & $[3]\rightarrow[2]$ & $(2,1,0,1,0)$ & $(0,0,c)$ & Table~\ref{tri4:tbl:31010} \\\hline                                
				$(1,2,0,0,1,0)$ & $[1,2,0,0]\rightarrow [3,1]$ & $(3,1,1,0)$ & $(1,1,0)$ & Table~\ref{tri4:tbl:3110}
				\\ \hline
				$(1,2,0,0,0,1,0)$ & $[1,2,0,0]\rightarrow [3,1]$ & $(3,1,0,1,0)$ & $(1,1,0)$ & Table~\ref{tri4:tbl:31010}
				\\ \hline                                
				
				$(1,2,0,1,0)$ & $[1,2,0,1]\rightarrow [3,2]$ & $(3,2,0)$ & $(1,1,0)$ & Table~\ref{tri4:tbl:320}
				\\ \hline	
				
				$(1,2,0,1,0)$ & $[1,2,0]\rightarrow[3,0]$ & $(3,0,1,0)$ & $(0,1,0)$ & Table~\ref{tri4:tbl:31010}	
				\\ \hline
				$(2,0,1,0,1)$ & $[2,0,1]\rightarrow[1,2]$ & $(1,2,0,1)$ & $(1,0,0)$ & Table~\ref{tri4:tbl:2110}
				\\ \hline
				$(2,0,0,1,0,1)$ & $[2,0,0]\rightarrow[1,1]$ & $(1,1,1,0,1)$ & $(1,0,0)$ &  \onlyones
				\\ \hline
				$(2,0,0,0,1,0,1)$ & $[2,0,0]\rightarrow[1,1]$ & $(1,1,1,0,1,0)$ & $(1,0,0)$ & \onlyones
				\\ \hline
				$(2,1,0,1,0)$ & $[2]\rightarrow[1]$ & $(1,1,0,1,0)$ & $(0,0,c)$ & \onlyones
			        \\ \hline

			\end{tabular}
		\end{center}
		\label{tri4:tbl:31010}
	\end{table}

	After considering all the tables, in all cases we have shown that we cannot reach the chain vector $(3,3)$, as desired.
	
\end{proof}

One of the key things going on in this proof is that starting with one of the chain vectors with fewer than four triangles and iteratively applying DTRs, any $2$ that appeared in a resulting chain vector was either next to another nonzero entry or was the only nonzero entry. In particular $(2,0,1,0,1,0)$ never arose in the proof of the theorem.  This is key because $(2,0,1,0,1,0)$ is a chain vector of a $K_5$ descendant. It is not obvious that the spacing zeros in the chain vector should be necessary for $K_5$ descendants, but that this is true is one of the consequences of the proof.

Generally, the structure of $K_5$ descendants is quite subtle and difficult.  Our results are only a beginning.

\clearpage

\bibliographystyle{plain}
\bibliography{main}

\begin{thebibliography}{10}

\bibitem{BrBe}
Prakash Belkale and Patrick Brosnan.
\newblock Matroids, motives, and a conjecture of {K}ontsevich.
\newblock {\em Duke Math. J.}, 116(1):147--188, 2003.
\newblock arXiv:math/0012198.

\bibitem{bek}
Spencer Bloch, H\'el\`ene Esnault, and Dirk Kreimer.
\newblock On motives associated to graph polynomials.
\newblock {\em Commun. Math. Phys.}, 267:181--225, 2006.
\newblock arXiv:math/0510011v1 [math.AG].

\bibitem{BrSinform}
David Broadhurst and Oliver Schnetz.
\newblock Algebraic geometry informs perturbative quantum field theory.
\newblock In {\em PoS}, volume LL2014, page 078, 2014.
\newblock arXiv:1409.5570.

\bibitem{bkphi4}
D.J. Broadhurst and D.~Kreimer.
\newblock Knots and numbers in $\phi^4$ theory to 7 loops and beyond.
\newblock {\em Int.J.Mod.Phys.}, C6(519-524), 1995.
\newblock arXiv:hep-ph/9504352.

\bibitem{Brbig}
Francis Brown.
\newblock On the periods of some {F}eynman integrals.
\newblock arXiv:0910.0114.

\bibitem{Brcosmic}
Francis Brown.
\newblock Feynman amplitudes, coaction principle, and cosmic {G}alois group.
\newblock {\em Commun.Num.Theor.Phys.}, 11:453--556, 2017.
\newblock arXiv:1512.06409.

\bibitem{BrS}
Francis Brown and Oliver Schnetz.
\newblock A {K3} in $\phi^4$.
\newblock {\em Duke Math J.}, 161(10):1817--1862, 2012.
\newblock arXiv:1006.4064.

\bibitem{BrS3}
Francis Brown and Oliver Schnetz.
\newblock Modular forms in quantum field theory.
\newblock {\em Communications in Number Theory and Physics}, 7(2):293 -- 325,
  2013.
\newblock arXiv:1304.5342.

\bibitem{Szigzag}
Francis Brown and Oliver Schnetz.
\newblock Single-valued multiple polylogarithms and a proof of the zig–zag
  conjecture.
\newblock {\em Journal of Number Theory}, 148:478--506, 2015.
\newblock arXiv:1208.1890.

\bibitem{BrSY}
Francis Brown, Oliver Schnetz, and Karen Yeats.
\newblock Properties of $c_2$ invariants of {F}eynman graphs.
\newblock {\em Advances in Theoretical and Mathematical Physics},
  18(2):323--362, 2014.
\newblock arXiv:1203.0188.

\bibitem{BrY}
Francis Brown and Karen Yeats.
\newblock Spanning forest polynomials and the transcendental weight of
  {F}eynman graphs.
\newblock {\em Commun. Math. Phys.}, 301(2):357--382, 2011.
\newblock arXiv:0910.5429.

\bibitem{CYgrid}
Wesley Chorney and Karen Yeats.
\newblock $c_2$ invariants of recursive families of graphs.
\newblock {\em Ann. Inst. Henri Poincar\'e Comb. Phys. Interact.}, (to appear).
\newblock arXiv:1701.01208.

\bibitem{ck0}
Alain Connes and Dirk Kreimer.
\newblock Hopf algebras, renormalization and noncommutative geometry.
\newblock {\em Commun. Math. Phys.}, 199:203--242, 1998.
\newblock arXiv:hep-th/9808042.

\bibitem{Dc2}
D.~Doryn.
\newblock The $c_2$ invariant is invariant.
\newblock arXiv:1312.7271.

\bibitem{D4face}
D.~Doryn.
\newblock Dual graph polynomials and a 4-face formula.
\newblock arXiv:1508.03484.

\bibitem{FlSe09}
Philippe Flajolet and Robert Sedgewick.
\newblock {\em Analytic combinatorics}.
\newblock cambridge University press, 2009.

\bibitem{FInvestc2}
Simone Hu, Oliver Schnetz, Jim Shaw, and Karen Yeats.
\newblock Further investigations into the graph theory of $\phi^4$-periods and
  the $c_2$-invariant.
\newblock arXiv:1812.08751.

\bibitem{iz}
Claude Itzykson and Jean-Bernard Zuber.
\newblock {\em Quantum Field Theory}.
\newblock McGraw-Hill, 1980.
\newblock Dover edition 2005.

\bibitem{patras}
J.~M.~Gracia-Bondia K.~Ebrahimi-Fard and F.~Patras.
\newblock A {L}ie theoretic approach to renormalization.
\newblock {\em Commun. Math. Phys.}, 276(2):519--549, 2007.
\newblock arXiv:hep-th/0609035.

\bibitem{code}
Mohamed Laradji.
\newblock dtd: Double triangle descendants.
\newblock https://github.com/mlaradji/dtd.

\bibitem{LaradjiDTDoK5}
Mohamed Laradji.
\newblock Double triangle descendants of {K}5.
\newblock Master's thesis, Simon Fraser University, 2017.
\newblock http://summit.sfu.ca/item/17812.

\bibitem{Lmod}
Adam Logan.
\newblock New realizations of modular forms in {C}alabi-{Y}au threefolds
  arising from $\phi^4$ theory.
\newblock arXiv:1604.04918.

\bibitem{Mmotives}
Matilde Marcolli.
\newblock {\em Feynman Motives}.
\newblock World Scientific, 2010.

\bibitem{MYlinreduc}
B.~{Moore} and K.~{Yeats}.
\newblock {Graph Minors and the Linear Reducibility of Feynman Diagrams}.
\newblock {\em ArXiv e-prints}, August 2017.

\bibitem{Snumbers}
Oliver Schnetz.
\newblock Numbers and functions in quantum field theory.
\newblock arXiv:1606.08598.

\bibitem{Sphi4}
Oliver Schnetz.
\newblock Quantum periods: A census of $\phi^4$-transcendentals.
\newblock {\em Communications in Number Theory and Physics}, 4(1):1--48, 2010.
\newblock arXiv:0801.2856.

\bibitem{SFq}
Oliver Schnetz.
\newblock Quantum field theory over $\mathbb{F}_q$.
\newblock {\em Elec. J. Combin.}, 18, 2011.
\newblock arXiv:0909.0905.

\bibitem{Yprefix}
Karen Yeats.
\newblock A study on prefixes of $c_2$ invariants.
\newblock arXiv:1805.11735.

\bibitem{Ycirc}
Karen Yeats.
\newblock A few $c_2$ invariants of circulant graphs.
\newblock {\em Commun. Number Theory Phys.}, 10(1):63--86, 2016.
\newblock arXiv:1507.06974.

\bibitem{Ycombpers}
Karen Yeats.
\newblock {\em A combinatorial perspective on quantum field theory}, volume~15
  of {\em SpringerBriefs in Mathematical Physics}.
\newblock Springer, Cham, 2017.

\end{thebibliography}

\end{document}